\documentclass[12pt,a4paper]
{amsart}
\usepackage{fullpage}

\usepackage[utf8]{inputenc}
\usepackage[all]{xy}
\usepackage{amsmath, amsthm, amssymb, amsfonts,mathrsfs} 
\usepackage[normalem]{ulem}
\usepackage{comment,upgreek} 

\usepackage[usenames]{xcolor}
\usepackage{hyperref} 
\usepackage{graphicx}
\usepackage{tikz-cd}

\usepackage[nameinlink,capitalise]{cleveref}
\usepackage{mathrsfs}

\newcommand{\SL}{\operatorname{SL}}

\newcommand{\C}{\mathbb{C}}
\newcommand{\Z}{\mathbb{Z}}

\newcommand{\F}{\mathbb F}
\newcommand{\Q}{\mathbb Q}
\newcommand{\Id}{\mathrm{Id}}

\newcommand{\bma}{\begin{pmatrix}}
\newcommand{\ema}{\end{pmatrix}}
\newcommand{\bsm}{\left(\begin{smallmatrix}}
\newcommand{\esm}{\end{smallmatrix}\right)}
\newcommand{\Tr}{\operatorname{Tr}}

\newcommand{\tor}{\operatorname{tor}}
\newcommand{\Tor}{\operatorname{Tor}}

\newcommand{\Frac}{\operatorname{Frac}}

\newtheorem{theorem}{Theorem}[section]

\newtheorem{corollary}[theorem]{Corollary}
\newtheorem{lemma}[theorem]{Lemma}
\newtheorem{proposition}[theorem]{Proposition}

\theoremstyle{definition}

\newtheorem{example}[theorem]{Example}

\newtheorem{remark}[theorem]{Remark}

\usepackage[pagewise, mathlines]{lineno} 
\usepackage{time, enumerate, bm}

\usepackage{etoolbox}          

\newcommand*\linenomathpatch[1]{%
  \cspreto{#1}{\linenomath}%
  \cspreto{#1*}{\linenomath}%
  \csappto{end#1}{\endlinenomath}%
  \csappto{end#1*}{\endlinenomath}%
}
\newcommand*\linenomathpatchAMS[1]{%
  \cspreto{#1}{\linenomathAMS}%
  \cspreto{#1*}{\linenomathAMS}%
  \csappto{end#1}{\endlinenomath}%
  \csappto{end#1*}{\endlinenomath}%
}

\expandafter\ifx\linenomath\linenomathWithnumbers 
 \let\linenomathAMS\linenomathWithnumbers
 \patchcmd\linenomathAMS{\advance\postdisplaypenalty\linenopenalty}{}{}{}
\else
  \let\linenomathAMS\linenomathNonumbers
\fi

\linenomathpatch{equation}
\linenomathpatchAMS{gather}
\linenomathpatchAMS{multline}
\linenomathpatchAMS{align}
\linenomathpatchAMS{alignat}
\linenomathpatchAMS{flalign}

\makeatletter
\patchcmd{\mmeasure@}{\measuring@true}{
  \measuring@true
  \ifnum-\linenopenaltypar>\interdisplaylinepenalty
    \advance\interdisplaylinepenalty-\linenopenalty
  \fi
  }{}{}
\makeatother

\newcommand{\ol}{\overline}

\newcommand{\wt}[1]{{\widetilde{#1}}}

\newcommand{\mf}[1]{{\mathfrak{#1}}}

\newcommand{\mca}[1]{{\mathcal{#1}}}
\newcommand{\bs}[1]{{\boldsymbol{#1}}}

\newcommand{\To}{\longrightarrow}

\newcommand{\surj}{\twoheadrightarrow}

\newcommand{\pmx}[1]{\begin{pmatrix}#1\end{pmatrix}}
\newcommand{\spmx}[1]{{\small \pmx{#1}}}

\makeatletter
\@namedef{subjclassname@2020}{
\textup{2020} Mathematics Subject Classification}
\makeatother 



\title[Non-acyclic $\SL_2$-representations of twisted Whitehead links]
{\large 
Multiplicity of non-acyclic $\SL_2$-representations and \\ $L$-functions of 
the odd-twisted Whitehead links}

\author{L\'{e}o B\'{e}nard} 
\email{leo.benard@univ-amu.fr}
\address{Aix Marseille Universit\'{e}, CNRS, I2M, Marseille, France, Site de Saint-Charles 3 place Victor Hugo, Case 19\\
13331 Marseille c\'{e}dex 3, France}

\author{Ryoto Tange} 
\email{rtange.math@gmail.com}
\address{Department of Mathematics, Rikkyo University\\
 3-34-1 Nishi-Ikebukuro, Toshima-ku, 171-8501, Tokyo, Japan}

\author{Anh T. Tran} 
\email{att140830@utdallas.edu}
\address{Department of Mathematical Sciences, The University of Texas at Dallas\\ 
Richardson, TX 75080, USA} 

\author{Jun Ueki} 
\email{uekijun46@gmail.com}
\address{Department of Mathematics, Faculty of Science, Ochanomizu University; 
2-1-1 Otsuka, Bunkyo-ku, 112-8610, Tokyo, Japan}

\begin{document}

\newenvironment{nouppercase}{%
\let\uppercase\relax%
\let\MakeUppercase\relax%
\renewcommand{\uppercasenonmath}[1]{}}{}

\begin{nouppercase}
\maketitle 
\end{nouppercase}


\begin{abstract} 
We study the divisor of the Reidemeister torsion on the variety of irreducible $\SL_2\C$-characters of certain knots and links, and provide a geometric interpretation of them. 
We focus in particular on the family of odd-twisted Whitehead links $W_{2n-1}$ and prove that these divisors have multiplicity two. 
Furthermore, we apply these results to the study of the $L$-functions of the universal deformations of representations over fields with characteristic $p>2$ of these link groups. 
\end{abstract} 

\tableofcontents

\section{Introduction}

\subsection{The divisor of the Reidemeister torsion} 

In this paper, we investigate a certain divisor on the variety of irreducible $\SL_2$-characters of a 3-manifold in a certain family, unifying the perspectives of our previous studies \cite{Benard2020OJM, TTU}. 

\emph{The ${\rm SL}_2\C$-character variety} $X(M)$ of a manifold $M$ is an algebraic variety, which may be seen as the space of the conjugacy classes of representations $\rho:\pi_1(M)\to {\rm SL}_2\C$ of the fundamental group. 
A representation $\rho:\pi_1(M)\to {\rm SL}_2\C$ is often \emph{acyclic}, that is, the homology $H_\ast(M,\rho)$ of the twisted complex $C_\ast(M,\rho)$ is trivial. 
For instance, if $M$ is a hyperbolic 3-manifold or a Seifert-fibered space, then 
the classes of acyclic representations form a non-empty Zariski open subset of  $X(M)$. 
If $\rho$ is acyclic, one can define a numerical invariant  $\tau_\rho(M)\in \C^\ast$ called \emph{the Reidemeister torsion}  of the pair $(M,\rho)$.
This extends to a rational function $\tau$ called \emph{the torsion function} on $X(M)$, and, on the subvariety $X^{\rm irr}(M)$ of irreducible characters, the zero locus of $\tau$ coincides with the set of the classes of \emph{non-acyclic representations}. 
Our main concern will be this zero locus with the information of multiplicity, namely, \emph{the divisor} of $\tau$. 

In \cite{Benard2020OJM}, the first author studied the case of a knot exterior in $S^3$, 
and exhibited a criterion for $\tau$ being a non-constant regular function on the affine variety. 
In particular, he gave a sufficient condition for the zero locus of $\tau$ being not void.

In \cite{TTU}, the last three authors studied the exteriors of \emph{twist knots} $J(2,2l)$ with $n\in \Z$ from this perspective. 
Every twist knot has a one-dimensional variety of irreducible ${\rm SL}_2\C$-characters, and the torsion function $\tau$ has a finite number of zeros there. 
Interestingly, they proved that \emph{all these zeros appear with vanishing multiplicity two} \cite[Theorem B]{TTU}. 
In addition, they suggested a geometric interpretation for non-acyclic representations. Namely, they proved that an irreducible ${\rm SL}_2\C$-representation of $\pi_1(S^3-J(2,2l))$ is non-acyclic if and only if it factors through the $-3$-Dehn surgery and the image of a specific element has order 3 \cite[Theorems C, D]{TTU}.  

In this paper, we extend these observations to explore theoretical backgrounds, further geometric interpretations, and arithmetic implications of the ``multiplicity two'' phenomenon. 

\subsection{The Whitehead link $W$ and the $(-3,-3)$-surgery} 
A clue will be the fact that every twist knot exterior $S^3-J(2,2l)$ is the result $W(-1/l)$ of the $-1/l$-filling of $S^3-W$ on a component of the Whitehead link $W$ (\cref{fig:Whitehead}). 
Note that the character variety and the torsion function of $W$ have been thoroughly studied in \cite{Tran2016JKTR.CV, Tran2018IJM, NguyenTran}.  

Let $W(-3,-3)$ denote the result of the $(-3,-3)$-surgery along $W$. 
Then the variety $X^{\rm irr}(W(-3,-3))$ of irreducible ${\rm SL}_2\C$-characters of $W(-3,-3)$ 
may be regarded as a 1-dimensional irreducible subvariety of $X^{\rm irr}(S^3-W)$. 
We will establish the following. 
\begin{theorem} \label{thm.W} 
\label{theo:main} 
The divisor $D(\tau_W)$ of the torsion function $\tau_W$ on the variety $X^{\rm irr}(S^3-W)$ of irreducible ${\rm SL}_2\C$-characters of the Whitehead link $W$ satisfies 
\[D(\tau_W)=2 \, X^{\rm irr}(W(-3,-3)).\] 
\end{theorem}
This means that the divisor of $\tau_W$  has multiplicity two, and that an irreducible ${\rm SL}_2\C$-representation of $\pi_1(S^3-W)$ is non-acyclic if and only if it factors through 
the homomorphism $\pi_1(S^3-W)\surj \pi_1(W(-3,-3))\cong \Gamma(3,3,\infty)$. 

Since this space $W(-3,-3)$ is homeomorphic to the connected sum $L(3,1)\# L(3,1)$ of two lens spaces, by verifying that the (non)-acyclicity behaves well under the fillings (Lemmas \ref{lem:MV}, \ref{lem:intermult}), we may deduce the following remarkable assertion. 

\begin{corollary} \label{cor:-3} 
All irreducible ${\rm SL}_2\C$-characters of $L(3,1)\# L(3,1)$ are non-acyclic. 
\end{corollary} 

\subsection{Exceptional fillings}
Note in addition that the result $W(-3)$ of the $-3$-filling along a component of $W$ is a Seifert fibered manifold with boundary (see also \cite[Table 1]{MP}) satisfying $\pi_1(W(-3))\cong (\Z/3 \ast \Z/3)\rtimes \Z$. 
Indeed, $W(-3)$ is a circle bundle over a 2-dimensional orbifold with two conical points of order 3 and one cusp, obtained by gluing two hyperbolic triangles with angles $(\pi/3,\pi/3,\infty)$ along their edges. 

The observation on $W$ leads us to prove the following theorem, which partially explains the general nature behind it. 
\begin{theorem} \label{thm.SF} 
Let $M$ be a 3-manifold with cusps, $N$ a Seifert fibered manifold obtained by a filling along a component of a boundary component of $M$. 
Let $\rho:\pi_1(M)\to {\rm SL}_2\C$ be an irreducible representation that factors through $\pi_1N$. 
If $\rho$ is non-trivial on the surgered component of $\partial M$ and $\rho$ sends a generic fiber to $I_2$, then $\rho$ is non-acyclic. 
\end{theorem} 
According to Thurston, when $M$ is hyperbolic, such a filling is \emph{exceptional}, in the sense that a generic filling yields a hyperbolic manifold. A generic fiber of $N$ is central in $\pi_1N$, so any irreducible representation $\rho$ sends a generic fiber to $\pm I_2$. 

We see that any non-acyclic representation of $\pi_1(S^3-W)$ factors through the $-3$-surgery on a component, and the $-3$-surgery on the other component kills the homotopy class of the fiber. 
So, any representation which factors through the $(-3,-3)$-filling satisfies the hypothesis of \Cref{thm.SF}. 

\subsection{Twist knots and once-punctured tori} 
We also prove that the multiplicity of zeros of the torsion function behaves well under a filling (\cref{lem:intermult}). By \Cref{thm.W}, we obtain the following. (See \Cref{ss.surgeries} for our convention of the fillings.) 
\begin{corollary} \label{cor:mult2Wr} 
Let $-3\neq r\in \Q\cup\{\infty\}$ and let $W(r)$ denote the result of the filling of $S^3-W$ along a component of $W$ with slope $r$. 
Then the divisor $D(\tau)$ of the torsion function on $X^{\rm irr}(W(r))$ has multiplicity at least two. 
\end{corollary} 
If $l\in \Z$, then $W(l)$ is a once-punctured torus bundle with tunnel number 1, which has been deeply studied in \cite{BakPet}, whereas $W(-1/l)=S^3-J(2,2l)$ is the exterior of a twist knot. 
The assertion without ``at least'' for twist knots \cite[Theorem B]{TTU} further implies that $X^{\rm irr}(W(-3,-3))$ and $X^{\rm irr}(S^3-J(2,2l))$ transversely intersect on $X^{\rm irr}(S^3-W)$. 
We may re-interpret \cite[Theorems C, D]{TTU} as well. 
Furthermore, we obtain the following. 

\begin{corollary} \label{cor.SU2}
Every non-acyclic irreducible representation $\rho$ of a twist knot $J(2,2l)$ factors through
the triangle group $\Gamma(3,3,|3l-1|)$. 
\end{corollary} 

\subsection{Odd-twisted Whitehead links $W_{2n-1}$}
We further study the family of odd-twisted Whitehead links $W_{2n-1}$ with $n\in \Z$ (see \cref{fig:TwistedW}). 
These links are hyperbolic and have linking numbers zero. 
Their ${\rm SL}_2\C$-character varieties $X(W_{2n-1})$ and the torsion functions $\tau$ have been computed by a series of studies by the third author and Nguyen \cite{Tran2016JKTR.CV, Tran2018IJM,NguyenTran}. 
In particular, each of the varieties has a specific component called \emph{the geometric component}, namely, the component containing the class of holonomy representations. 

We first verify the following with slight effort. 
\begin{proposition}
\label{prop:smooth}
Each component of the $\SL_2\C$-character variety of the odd-twisted Whitehead link $W_{2n-1}$ is a smooth surface in $\C^3$.
\end{proposition}
Afterward, we will establish the following. 
\begin{theorem}
\label{theo:TW}
Let $D(\tau)$ denote the divisor of the torsion function $\tau$ on the variety $X^{\rm irr}(W_{2n-1})$ of irreducible $\SL_2\C$-characters of the odd-twisted Whitehead link $W_{2n-1}$ with any $n\in \Z$. 
Then, $D(\tau)$ has multiplicity two on the geometric component, whereas $D(\tau)$ has multiplicity one or $D(\tau)=0$ holds on any other component. 
\end{theorem}

For each $l\in \Z$, the $-1/l$-filling of $S^3-W_{2n-1}$ along a component of $W_{2n-1}$ yields the exterior 
$S^3-J(2n,2l)$ of so-called the double twist knot $J(2n,2l)$. 
Similarly to the case of $W$ and $J(2,2l)$, we may deduce the following. 
\begin{corollary}
\label{coro:doubletwist}
The torsion vanishes with multiplicity at least two on the variety of irreducible ${\rm SL}_2\C$-characters of any double twist knot $J(2n,2l)$.
\end{corollary}

\begin{remark} 
(1) We have similar results for even-twisted Whitehead links $W_{2n}$ and double twist knots of type $J(2n-1,2l)$. 
The divisor $D(\tau)$ of the torsion function has multiplicity two on the geometric component of $X^{\rm irr}(W_{2n})$, whereas $D(\tau)$ has multiplicity one on the other components. 
Thus, $X^{\rm irr}(W_{2n})$ has no component where the torsion function vanishes identically. 

(2) For every $k \in \Z$, the Whitehead link $W_k$ is the result of a filling along a component of the Borromean ring $B=6^3_2=L6a4$. 
As we derived properties of $J(2,2l)$ from $W$ simultaneously, we may study the families $(W_{2n-1})_n$ and $(W_{2n})_n$ from the viewpoint of 3-component links. 

Detailed arguments will be given elsewhere. 
\end{remark}

\subsection{A remark on ${\rm SL}_3$-deformation} 
The multiplicity of the divisor of the torsion is also related to the singularities of the $\SL_3$-character variety, as we explain now. 

The $\SL_2$-character variety of a knot exterior always contains a line of \emph{reducible} representations. This line may intersect the other components, in which irreducible representations are dense. 
It has been known since Burde and de Rham \cite{Burde1967MathAnn, deRham1967EnseignMath} that the intersection points between the components of reducible representations and the other components correspond to roots of the Alexander polynomial of the knot. 

This has been generalized by Heusener and Porti \cite{HP15}. They proved the following: if an $\SL_n\C$-representation lies at an intersection point between a component of reducible representations and a component containing irreducible representations, then a certain \emph{twisted} Alexander polynomial has some specific zeros.
Moreover, if the zero is simple, 
then the intersection is as mild as possible, namely, they intersect transversely.   

In our setting, the torsion of a representation $\rho \colon \pi_1(M) \to \SL_2\C$ is the value at $t=1$ of the $\rho$-twisted Alexander polynomial $\Delta_\rho(t)$, and the multiplicity two phenomenon that we observe means that this yields a double zero of this twisted Alexander polynomial.
Heusener--Porti's result, combined with our study, suggests that the curve of non-acyclic representations of the twisted Whitehead links corresponds to the classes of $\SL_3\C$-representations of the form
 $\bma \rho & \ast \\ 0 & 1 \ema$ that are highly singular intersection points in their respective $\SL_3$-character varieties.  
 It seems to fit well with the few explicit commutations of $\SL_3(\C)$-character varieties available (figure-eight knot and Whitehead link, see \cite{HMP, GuillouxWill}).


\subsection{$L$-functions} \label{ss.Lfn}
In the final section of this article, we pursue the study of the $L$-functions of universal deformations, raised in a viewpoint of number theory \cite{Mazur2000, MTTU2017, KMTT2018, TTU} (see also \cite{MorishitaTerashima2007, Morishita2012, KMTT2023}), to interpret the ``multiplicity two'' phenomenon into the property of the $L$-functions of the odd-twisted Whitehead links $W_{2n-1}$. 

Let $\overline \rho \colon \pi \to \SL_2\F$ be a representation over a field $\F$ with characteristic $p>2$ and let $\mathcal O$ be a complete discrete valuation ring (CDVR) with the residual field $\mathcal O/ \mathfrak m_{\mathcal O} = \F$. A typical example is $(\F,\mathcal O)=(\F_p,\Z_p)$. 
\emph{The universal deformation} ${\bs \rho}$ of $\ol{\rho}$ over $\mca O$ is a lift of $\ol{\rho}$ 
to a representation over a complete local $\mca O$-algebra $\mca R_{\ol{\rho}}$ satisfying the following universal property: 
any lift of $\ol{\rho}$ to any complete local $\mca O$-algebra uniquely factors through ${\bs \rho}$ up to strict equivalence (conjugation fixing the residual representation $\ol \rho$).

Such a representation $\bs \rho$ is unique up to strict equivalence. 
Considering the homology of the complex of $\mathcal R_{\ol{\rho}}$-modules given by $\bs \rho$,  the analogy between knots and prime numbers designed as \emph{arithmetic topology} suggests that the order $L_{\bs \rho}$ of the first homology group of this complex should be of deep interest. 
This $L_{\bs \rho}$ is called \emph{the $L$-function of} ${\bs \rho}$, which is actually determined by $(\ol \rho, \mca O)$, 
and has been studied by many authors (see the references above).  

Let $f_n\in \Z[x,y,z]$ denote the reduced polynomial defining the geometric component of the ${\rm SL}_2\C$-character variety and put $\overline{f_n}:=f_n \ {\rm mod}\ p$. 
Note that the ${\rm SL}_2\C$-Reidemeister torsion function is given by $\tau_n\in\Z[x,y,z]$ by \cite[Corollary 2]{NguyenTran}, see \Cref{subsec:vanish}. 
Since the multiplicities of the zeros of the $L$-function $L_{\bs \rho}$ are intimately related to the vanishing multiplicity of the Reidemeister torsion, we may deduce the following.

\begin{theorem} \label{theo:L} 
Let $\ol{\rho}:\pi\to \SL_2\F$ be an absolutely irreducible non-acyclic representation of the group of the twisted Whitehead link $W_{2n-1}$ 
corresponding to a smooth point $(\ol{A},\ol{B},\ol{C}) \in \F^3$ of the geometric component $\ol{f_n}=0$ of the character variety. 

Assume $\partial_z \ol{f_n}(\ol{A},\ol{B},\ol{C}) \neq 0$.  
Let $(A,B)\in \mca O^2$ be any lift of $(\ol{A},\ol{B})$ and let $z_{f_n}(x,y)\in \mca O[\![x-A, y-B]\!]$ denote the implicit function given by Hensel's lemma. 

Then, the $L$-function $L_{\bs \rho}$ of the universal deformation satisfies 
\[L_{\bs \rho} \doteq \tau_{n}(x,y,z_{f_n}(x,y))\] 
in $\mca{R}=\mca O[\![x-A, y-B]\!]$.  

Furthermore, the zeros of $L_{\bs \rho}$ have multiplicity two in the sense that  
$L_{\bs \rho} \in (x-A,y-B)^2$ and $L_{\bs \rho} \not\in (x-A,y-B)^3$ hold in $\ol{\Frac\mca O}[\![x-A,y-B]\!]$. 
\end{theorem} 

\subsection*{Organization of the paper}
In \cref{sec:pre}, we introduce character varieties, Reidemeister torsion, and the notion of the multiplicity of a divisor. We prove some preparatory lemmas and \cref{thm.SF} as well.  
In \cref{sec:Whitehead}, we treat the case of the Whitehead link, proving \cref{theo:main} and corollaries. 
In \cref{sec:tW}, we study the family of odd-twisted Whitehead links to prove \cref{prop:smooth} and \cref{theo:TW}. 
In \cref{sec:L}, we discuss universal deformations and prove \cref{theo:L} on the $L$-functions.

\subsection*{Acknowledgments}
The authors would like to express their sincere gratitude to Marco Maculan for very enlightening conversations on the notion of the multiplicity of a divisor. 
We also thank Rapha\"el Alexandre, Elisha Falbel, Antonin Guilloux, Michael Heusener, Tomoki Mihara, and Gauthier Ponsinet for useful discussions and references. 

This paper has been written while L.\,B. was partially funded by the Research Training Group 2491 ``Fourier Analysis and Spectral Theory'', University of G\"ottingen.
A.\,T. has been supported by a grant from the Simons Foundation (\#708778). 
J.\,U. has been partially supported by JSPS KAKENHI Grant Numbers JP19K14538 and JP23K12969. 

\section{Preliminaries} \label{sec:pre}
This section contains preliminary materials on character varieties (\cref{subsec:prechar}), twisted homology (\cref{subsec:homol}), Reidemeister torsion (\cref{subsec:pretors}), torsion functions (\cref{ss.torsion}), 
acyclicity and fillings (\cref{subsec:surg}), surgeries (\cref{ss.surgeries}), 
Seifert fibered manifolds \cref{ss.SF}, 
and the multiplicity of a divisor (\cref{subsec:mult}).  
In \cref{subsec:surg}, we prove some technical lemmas. In \cref{ss.SF}, we establish \cref{thm.SF}. 

\subsection{Character variety}
\label{subsec:prechar} 
Let $R$ be any integral domain with characteristic zero and let ${\rm Frac}R$ denote the field of fractions. 
(In Sections 3 and 4, our concern will be the cases with $R=\C$.) 
Let $\Gamma$ be a finitely generated group and consider the set ${\rm Hom}(\Gamma,{\rm SL}_2R)$ of group homomorphisms. An element of ${\rm Hom}(\Gamma,{\rm SL}_2R)$ is called a \emph{representation}. 
\emph{The ${\rm SL}_2R$-character variety} of $\Gamma$ over $R$ is defined by 
\[ X_R(\Gamma) := \{\Tr \rho\mid \rho\in {\rm Hom}(\Gamma, \SL_2R)\},\]
where ${\rm Tr}$ denotes the trace. 
This set may be seen as an algebraic set, unique up to canonical isomorphisms, and is given as the zero locus of a family of polynomials. 
This is equivalent to considering the so-called GIT quotient 
${\rm Hom}(\Gamma,{\rm SL}_2R) /\!/{\rm SL}_2 R:={\rm Spec} (R[{\rm Hom}(\Gamma, \SL_2 R)]^{\SL_2 R})$, 
since the trace functions are the only invariant functions there. 

Let $\rho:\Gamma\to {\rm SL}_2 R$ be a representation and regard $R^2$ as a left $R[\Gamma]$-module via $\rho$. 
Then $\rho$ is said to be \emph{irreducible} if $R^2$ has no $R[\Gamma]$-submodule other than 0 and $R^2$, 
and it is said to be \emph{reducible} otherwise.  
In addition, $\rho$ is said to be \emph{absolutely irreducible} if it is irreducible over an integral closure $\ol{R}$, 
and it is said to be \emph{absolutely reducible} otherwise. 
Since $R^2$ is irreducible as a $R[\Gamma]$-module if and only if $({\rm Frac}R)^2$ is irreducible as a $({\rm Frac}R)[\Gamma]$-module, \cite[Lemma 1.2.1]{CS83} and \cite[Proposition 1.5.2]{CS83} yield the following lemma, assuring that 
these notions are well-defined for elements of $X_R(\Gamma)$. 

\begin{lemma} \label{lem:red} 
For a representation $\rho\colon \Gamma \to {\rm SL}_2 R$ with non-abelian image, 
the following are equivalent. 
\begin{itemize}
\item[(i)] $\rho$ is reducible. 

\item[(ii)] $\rho$ is reducible over an integral closure $\ol{R}$. 

\item[(iii)] $\Tr \rho (\gamma)=2$ holds at every element $\gamma$ of the commutator subgroup of $\Gamma$. 
\end{itemize} 
\end{lemma}

\begin{lemma} 
\label{prop:repequiv} 
If $\rho$ and $\rho'$ are in ${\rm Hom}(\Gamma,{\rm SL}_2 R)$ and $\rho$ is absolutely irreducible, then 
$\rho$ and $\rho'$ are equivalent over an algebraic closure $\ol{{\rm Frac} R}$   
if and only if ${\rm Tr}\rho={\rm Tr}\rho'$ holds. 
\end{lemma} 
The set $X_R^{\rm{a.i.}}(\Gamma)$ of absolutely irreducible characters is a Zariski open subset. 
We may write $X_{R}(\Gamma) = X_R^{\rm{a.r.}}(\Gamma) \sqcup X_R^{\rm{a.i.}}(\Gamma)$. 

There is a bijective correspondence between the equivalent classes of absolutely irreducible representations and points of $X_R^{\rm{a.i.}}(\Gamma)$, as well as that between the classes of indecomposable representations and points of $X_{R}(\Gamma)$. 

For a manifold $M$, we define $X_R(M)=X_R(\pi_1(M))$ and write $X_\C(M)=X(M)=X^{\rm irr}(M)\sqcup X^{\rm red}(M)$.

\subsection{Twisted complex and homology} \label{subsec:homol} 

Let $\rho\colon \Gamma \to {\rm SL}_2R$ be a representation. 
Then \emph{the twisted complex} of $(\Gamma,\rho)$ is defined by 
\[C_\ast(\Gamma,\rho):=R^2\otimes_{\Z[\pi_1(M)]} C_\ast(\Gamma),\]
where $C_\ast(\Gamma)$ is a free resolution of $\Z$ of left $\Z[\Gamma]$-modules, and 
we regard $R^2$ as a right $\Z[\pi_1(M)]$-module via $\rho\circ(g\mapsto g^{-1})$. 
The twisted homology $H_\ast(\Gamma,\rho):=H_\ast(C_\ast(\Gamma,\rho))$ is independent of the choice of $C_\ast(\Gamma)$. 

Let $M$ be a compact 3-manifold and $\rho\colon \pi_1(M) \to {\rm SL}_2R$ a representation. 
Let $C_\ast(M)$ be a CW complex of $M$ and let $C_\ast(\wt{M})$ denote its lift to the universal cover $\wt{M}\to M$. 
Then \emph{the twisted complex} of $(M,\rho)$ is defined by 
\[C_\ast(M,\rho):=R^2\otimes_{\Z[\pi_1(M)]} C_\ast(\wt{M}),\] 
where we regard $C_*(\widetilde M)$ as a left $\Z[\pi_1(M)]$-module via the deck transformation. 
The twisted homology $H_\ast(M,\rho):=H_\ast(C_\ast(M,\rho))$ is independent of the choice of $C_\ast(M)$. 

When $M$ is fixed, 
$\rho$ is said to be \emph{acyclic} if $H_i(M,\rho)=0$ holds for every $i$, and 
$\emph{non-acyclic}$ otherwise. 
In addition, $\rho$ is said to be \emph{rationally} (\emph{non}-)\emph{acyclic} if $\rho\otimes_R {\rm FracR}$ is (non-)acyclic. 
These notions also turn out to be well-defined over $X_R(M)$. 

Since the Euler characteristics satisfy $\sum_i (-1)^i{\rm dim} H_i(M,\rho\otimes {\rm Frac}R) =\chi(C_*(M,\rho))={\rm deg}\,\rho \cdot \chi(M)$, we have a rationally acyclic representation $\rho$ only if $\chi(M)=0$ holds, e.g., $M$ is the exterior of a link in a rational homology 3-sphere. 

If $M$ is an Eilenberg--MacLane space with $\Gamma=\pi_1(M)$, then 
these complexes are homotopic, and Poincar\'{e}--Hopf's theorem assures that 
$H_\ast(M,\rho)\cong H_\ast(\Gamma,\rho)$. 

\subsection{Reidemeister torsion}
\label{subsec:pretors} 
Let $C_\ast=(C_i,\partial_i)_i$ be a finite complex over an integral domain $R$
that is rationally acyclic, and let $\mf{c}_i$ be an $R$-basis of $C_i$ for each $i$.
Consider the exact sequence 
\[0 \to Z_i \to C_i \xrightarrow{\partial_i} B_{i-1} \to 0.\]
Take an $R$-basis $\mf{b}_i$ of $B_i = Z_i$  and a lift $\wt{\mf{b}}_i$ to $C_{i+1}$. 
Then $\mf{b}_i\sqcup \wt{\mf{b}}_{i-1}$ is also a basis of $C_i$. 
Let $[\mf{b}_i \sqcup \wt{\mf{b}}_{i-1} \colon \mf{c}_i]$ denote the determinant of the transition matrix. 
Then \emph{the Reidemeister torsion} (or \emph{the torsion} for short) of $(C_\ast,\mf{c}_\ast)$ is defined by the alternating product 
\[\tor(C_*, \mf{c}_*)  := \prod_i [\mf{b}_i \sqcup \wt{\mf{b}}_{i-1} \colon \mf{c}_i]^{(-1)^i} \in ({\rm Frac}R)^\ast.\]
This value does not depend on the choices of $(\mf{b}_i)_i$ and $(\wt{\mf{b}}_i)_i$, 
and is known to be a homotopy invariant. 

Let $M$ be a connected compact 3-manifold and $\rho \colon \pi_1(M) \to {\rm SL}_2 R$ a rationally acyclic representation. 
Let $C_\ast(M)$ be a CW complex of $M$ with a basis $\mf{c}_\ast=(\mf{c}_i)_i$. 
Let $\mf{c}_i=(c_i^{\,1},\ldots,c_i^{\,n_i})$ and choose a lift $\wt{c}_i^{\,j}\in C_i(\wt{M})$ of each $c_i^{\,j}$. 
Let $(v_1,v_2)$ be a basis of $R^2$. 
Then 
\[\wt{\mf{c}}_i\otimes \rho:=(v_1\otimes \wt{c}_i^{\,1}, v_2\otimes \wt{c}_i^{\,1}, \ldots, v_1\otimes \wt{c}_i^{\,n_i}, v_2\otimes \wt{c}_i^{\,n_i})\] 
becomes a basis of $C_i(M,\rho)$. 
We define the torsion of $(M,\rho)$ by 
\[\tor_M(\rho):=\tau(C_*(M, \rho),\mf{c}_i\otimes \rho).\]  
A deep theorem (cf.\,\cite{Chapman, Cohen}) ensures that this value is independent of the CW structure, and the choice of basis only affects the sign. 

In general, the torsion may contain information that $\pi_1(M)$ does not. 
For instance, lens spaces are completely determined by their fundamental groups together with their torsions, but not only by their fundamental groups. 

If $M$ is a connected compact 3-manifold with non-trivial boundary, then it collapses to a 2-dimensional complex, so its torsion may be calculated as the torsion of the complex of group homology of $\pi_1(M)$ given by the Fox free derivative with a twisted coefficient. 


\subsection{Torsion function} 
\label{ss.torsion}

The following might be well-known. 

\begin{proposition} 
\label{prop.nonacyclic} 
Let $M$ be a connected compact 3-manifold  
and $X$ a component of the character variety $X_R(M)$ with ${\rm dim}X\geq 1$. 

{\rm (1)} If $X$ contains a rationally acyclic representation, then the torsions of acyclic representations define a rational function $\tau$ on $X$. 

{\rm (2)} 
If in addition $\rho:\pi_1(M)\to {\rm SL}_2R$ is a representation on $X$ such that $\rho$ is absolutely irreducible,   
then $\rho$ is rationally non-acyclic if and only if $H_1(M,\rho\otimes_R {\rm Frac}R)\neq 0$. 

{\rm (3)} 
Further suppose that $\partial M\neq \emptyset$. 
Then $\rho$ is rationally non-acyclic if and only if 
its character $\Tr \rho$ corresponds to a zero of this function $\tau$. 
\end{proposition} 

\begin{proof} 
{\rm (1)} 
Once a pair $(C_\ast(M),\mf{c}_\ast(M))$ is fixed, the value $\tor_M(\rho)$ depends only on the conjugacy class of $\rho$. 
One should note that reducible representations $\rho,\rho'$ that belong to distinct conjugacy classes may satisfy $\tau(M,\rho)=\tau(M,\rho')$, 
since only the diagonal entries affect the calculation. 

Now let $X$ be a component of $X_R(M)$. 
Let $R[X]$ denote the ring of functions over $X$, which is isomorphic to the quotient of the universal character ring 
$\mca{B}[X]=R[\tau_g;\ g\in \pi_1(M)]/(\tau_e-2,\tau_g\tau_h-\tau_{gh}-\tau_{gh^{-1}};\ g,h\in \pi_1(X))$ 
by the prime ideal corresponding to the component $X$. 
Then we have so-called \emph{a tautological ${\rm SL}_2$-representation} $\rho^{\rm taut}$ over a certain quadratic extension of $R[X]$,  
and $\rho^{\rm taut}$ yields a representation $\rho$ at each point of $X$ by evaluation (cf.\,\cite[Subsection 1.2]{Benard2020OJM}, \cite[Subsection 3.3]{Marche-RIMS2016}). 

By the assumption that $X$ has a rationally acyclic representation, we see that $\rho^{\rm taut}$ is rationally acyclic. 
Since the torsion $\tau=\tau(M,\rho^{\rm taut})$ is determined only by the trace, 
$\tau$ belongs to the invariant field ${\rm Frac}(R[X])$. 
If $\rho$ is acyclic, then $\tau$ yields $\tau(M,\rho)$ by evaluation. 
Since $\tau$ is a non-zero rational function, the values at rationally acyclic representations determine $\tau$. 

(2) Suppose that $\rho$ is absolutely irreducible. 
Then, by \cref{lem:red}, there exists some $\gamma_0\in\pi_1(M)$ with $\Tr \rho(\gamma_0)\neq 2$. 
This implies that $H_0(M,\rho\otimes_R {\rm Frac}R)=0$. 
Indeed, $H^0(M, \rho)$ is isomorphic to the space $\{v \in ({\rm Frac}R)^2 \mid \forall \gamma\in \pi_1(M), \, \rho(\gamma)^{-1} v = v \}$ of invariant vectors,
where the latter is trivial since $\rho(\gamma_0)$ has no fixed non-zero vector, and 
the universal coefficients theorem yields $H_0(M,\rho\otimes_R {\rm Frac}R)=0$.
By the Poincar\'{e}--Lefschetz duality and the universal coefficient theorem, $\rho$ is non-acyclic if and only if $H_1(M,\rho\otimes_R {\rm Frac}R)\neq 0$.

(3) 
By the existence of a rationally acyclic representation, we have $\chi(M)=0$. 
By $\partial M\neq \emptyset$, a complex of $M$ collaspes to a finite 2-dimensional complex 
corresponding to a presentation $\pi_1(M)=\langle a_1,\ldots, a_s\mid r_1,\ldots, r_{s-1}\rangle$ with deficiency 1. 
Thus, we obtain a twisted complex of $\rho^{\rm taut}$ in the form 
\[0\To C_2\overset{d_2}{\To} C_1 \overset{d_1}{\To} C_0\To 0,\]
where, by the map $\Phi:R[\pi_1(M)]\to M_2(R);\, \sum c_g g\mapsto \sum c_g \rho^{\rm taut}(g)$, 
these maps $d_2$ and $d_1$ are presented by matrices $A=(\Phi(\partial r_i/\partial a_j))$ and $(\Phi(a_j-1))$ respectively. 
Let $A_j$ denote the result of deleting the $j$-th column in $A$. 
As is well-known, we may choose $k$ such that $\tau \doteq {\rm det} A_k/\Phi(a_k-1)$ holds.

By $\Tr \rho(\gamma_0)\neq 2$, we have that $\Phi(a_k-1)\neq 0$. Thus, we have 
$\tau(\rho)=0$ if and only if $d_2$ at $\rho$ has a small rank and $H_1(\rho\otimes {\rm Frac}R)\neq 0$. This completes the proof.
\end{proof} 

An alternative proof of \Cref{prop.nonacyclic} may be given by considering the value of the torsion $\Delta_M(t_1,\ldots,t_d)$ of the tensor representation $\rho\otimes \alpha:\pi_1(M)\to \SL_2(R[t_1^{\pm1},\ldots,t_d^{\pm 1}])$ at $t_1=\cdots=t_d=1$, 
where $\alpha$ denotes the representation induced by the abelianization map (cf.\,\cite[Subsection 2.3]{NguyenTran}). 

\cite[Remark 3.10]{Marche-RIMS2016} gives a decomposable tautological representation on the variety of reducible characters. We may also construct an indecomposable tautological representation there (cf.\,\cite{SaitoK1996}, \cite[Appendix]{Benard2018PhD}).

Since the Poincar\'{e}--Hopf duality asserts $H_i(M,\rho)\cong H_i(\pi_1(M),\rho)$, 
we see that the multiplicities of zeros of the torsion function are determined only by $\pi_1(M)$. 
Nevertheless, arithmetic topology suggests that this divisor should have some subtle nature of deep interest, and we will actually find a geometric interpretation in relation to certain surgeries, extending our previous observations. 

\subsection{Acyclicity and fillings}
\label{subsec:surg} 
We prove two technical lemmas for the cases with $R=\C$. 
\begin{lemma}
\label{lem:MV} 
Let $M$ be a compact $3$-manifold, $K$ a knot in $M$, and $\gamma\in \pi_1(M)$ a homotopy class of $K$. Let $\rho \colon \pi_1(M) \to {\rm SL}_2\C$ be a representation such that $\Tr \rho(\gamma) \neq 2$. 
Let $X_K$ denote the exterior of $K$ in $M$ and let $\varrho\colon \pi_1(X_K) \to {\rm SL}_2\C$ denote the composition of $\rho$ and $\pi_1(X_K)\surj \pi_1(M)$. Then $\varrho$ is acyclic if and only if $\rho$ is acyclic.
\end{lemma}

\begin{proof} 
Let $V_K$ be a tubular neighborhood of $K$, which is homeomorphic to a solid torus $D^2 \times S^1$, 
and let $T_K$ denote the boundary $\partial V_K$. 
Note that $\pi_1(V_K)$ is freely generated by $\gamma$, and let $c_\gamma \in \pi_1(T_K)$ any element such that $\{\gamma, c_\gamma\}$ generates $\pi_1(T_K)$.
Then we may find a complex $C_*(V_K, \rho)$ that is identified with 
\[ 0\To \C^2 \xrightarrow{I_2-\rho(\gamma)} \C^2 \To 0,\]
and we have $H_*(V_K, \rho) = 0$ by $\Tr \rho(\gamma)\neq 2$. 

Next, $C_\ast(T_K,\rho)$ may be identified with 
\[0\To \C^2 \xrightarrow{d_2} \C^4 \xrightarrow{d_1} \C^2\To 0\] 
where the maps are 
\[d_2 = \spmx{I_2-\rho(\gamma) \\ I_2-\rho(c_\gamma)},\ \ 
d_1 = \spmx{\rho(c_\gamma) - I_2 & \rho(\gamma)-I_2}.\] 
Again by $\Tr \rho(\gamma)\neq 2$, we see that $d_2$ is injective and $d_1$ is surjective, 
so we have $H_2(T_K,\rho)=H_0(T_K,\rho)=0$. 
Noting that the Euler characters $\chi(\bullet)$ satisfies $\chi(C_\ast(T_K),\rho)={\rm deg}\,\rho \cdot \chi(T_K)=0$, we obtain $H_1(T_K,\rho)=0$ as well. 

Now we use the Mayer--Vietoris sequence for the decomposition
\[M = X_K \cup_{T_K} V_K\]
and it comes that $H_*(X_K, \varrho) \cong H_*(M, \rho)$. The lemma follows.
\end{proof}

\begin{lemma}
\label{lem:open}
Let $M$ be a compact 3-manifold. Then being acyclic is an open property on the closure $\ol{X^{\rm irr}}(M)$ of the set of irreducible characters in $X_\C(M)$. 
\end{lemma}

\begin{proof} 
Assume that $M$ has an acyclic representation, so $\chi(M)=0$.  
We shall prove that for any irreducible acyclic representation $\varrho$, there exists an open neighborhood consisting of acyclic representations. 
Note that irreducible representations form an open dense subset of $\ol{X^{\rm irr}}(M)$, so there exists an open neighborhood of $\varrho$ such that every $\rho$ is irreducible there. 
By \cref{prop.nonacyclic} (2), in this neighborhood, $\rho$ is acyclic if and only if $H_1(M, \rho) = 0$ holds. 
Let us prove that $H_1(M, \rho) = 0$ is an open condition on $X_\C(M)$. 
Since $H_0(M, \rho) = 0$, the dimension of the kernel of the map $d_1$ is constant, and $\dim H_1(M, \rho)$ is determined by the rank of the matrix $d_2$. 
This matrix varies polynomially with the entries of the representation $\rho$. So, the rank of this matrix is a lower-semicontinuous function, and cannot decrease in a small neighborhood of $\varrho$. This completes the proof. 
\end{proof}

\subsection{Conventions for surgeries} 
\label{ss.surgeries} 
Let $L = K_1 \cup \ldots \cup K_s$ be a link (with an orientation) in an integral homology 3-sphere $M$ and let $V_L=V_{K_1}\cup \cdots \cup V_{K_s}$ be its tubular neighborhood. 
For each boundary torus $T_{K_i} := \partial V_{K_i}$, there is a natural pair $(m_i,l_i)$ of peripheral curves, that is, a pair of simple closed curves on $T_{K_i}$ generating $H_1(T_{K_i})$ such that, $m_i$ is anti-clockwise to $K_i$, 
the map $H_1(T_{K_i})\to H_1(M- K_i)\cong \Z$ sends $m_i\mapsto 1$ and  $l_i\mapsto 0$, 
and the map $H_1(T_{K_i})\to H_1(V_{K_i})\cong \Z; [K_i]\mapsto 1$ sends $m_i\mapsto 0$ and $l_i\mapsto 1$. 
Then the intersection number becomes $\iota(m_i,l_i)=1$. 

Let $(a,b)$ be a pair of coprime integers, so we have $a/b\in \Q\cup \{\infty\}$. 
Then the result of the $a/b$-surgery along $K_i$ is the $3$-manifold obtained by gluing a solid torus $S^1 \times D^2$ to the exterior $X_{K_i}=M-{\rm Int}(V_{K_i})$ of $K_i$ by an orientation reversing homeomorphism of $\partial (S^1 \times D^2) \cong T_{K_i}$ such that the curve $c = \{ \ast \} \times \partial D^2$ is identified with a curve presenting $a m_i + b l_i$ in $H_1(T_{K_i})$. 

\subsection{Seifert fibered manifolds} \label{ss.SF} 
Here, we prove \cref{thm.SF}. 

\begin{proposition} \label{prop.SF}
Let $N$ be a Seifert fibered 3-manifold and $\rho:\pi_1(N)\to {\rm SL}_2\C$ an irreducible representation. 
Then the twisted complex $C_*(N, \rho)$ is acyclic if and only if the image of the generic fiber is non-trivial. 
\end{proposition} 

\begin{proof} 
Kitano \cite[Main Theorem]{Kitano1994TJM} asserts the claim for compact Seifert fibered spaces without boundaries. 
The core of the proof is to describe $N$ as the union of tubular neighborhoods of singular fibers and a trivial bundle $N_0=\Sigma\times S^1$, $\Sigma$ being a surface, consider the Mayer--Vietoris long exact sequence, 
and deduced that $C_*(N, \rho)$ being acyclic is equivalent to $H_1(N_0,\rho)=0$ and to ${\rm det}(I-\rho(h))\neq 0$ for a generic fiber $h$.
Now it is clear that the same proof works verbatim in the case with a boundary.  
\end{proof}

\begin{proof}[Proof of \cref{thm.SF}] 
Let $\rho \colon \pi_1(M) \to \SL_2\C$ be as in the assertion, so there exists $\gamma \in \pi_1(\partial M)$ with $\rho(\gamma)$ non-trivial, and furthermore $\rho$ sends the generic fiber of $N$ to $\Id$.
Hence by \cref{prop.SF}, we see that $\rho:\pi_1(N)\to {\rm SL}_2\C$ in non-acyclic. 
Sine $\rho$ is not identity on $\gamma$, \cref{lem:MV} assures that $\rho:\pi_1(M)\to {\rm SL}_2\C$ is also non-acyclic. 
\end{proof} 

\subsection{Definition of multiplicity} \label{subsec:mult} 

The notion of the multiplicity of divisors has a long history and is closely related to the intersection multiplicity. 
Here we follow \cite[Chapter 3]{Shaf} to recollect these notions in our setting.

Let $X$ be a smooth algebraic set in the affine space $\C^n$ 
(which will be the case in this paper, see \cref{prop:smooth}). It is defined as the zero locus of a family of polynomials $p_1, \ldots, p_k \in \C[x_1, \ldots, x_n]$. The \emph{coordinate ring} of $X$ is $\C[X] = \C[x_1, \ldots x_n] /(p_1, \ldots p_k)$. It is the ring of polynomial functions defined on $X$. 
Let $P\in \C[X]$ and let $X\cap \{P=0\}=C_1 \cup \ldots \cup C_r$ denote the unique irreducible decomposition. 
Then \emph{the divisor} of $P$ on $X$ is the formal sum 
\[D_X(P)=\sum_{i=1}^r m_i C_i\] where $m_i\in \Z$ 
is \emph{the multiplicity} of $D_X(P)$ along $C_i$,  defined as follows. 
For each $i$, let $f_i\in \C[X]$ be an irreducible element such that $C_i=\{f_i=0\}$. 
The local ring $\mca O_i$ of $X$ along $C_i$ stands for the localization $\C[X]_{(f_i)}$ of $\C[X]$ at the prime ideal $(f_i)$. 
We define $m_i$ to be the integer satisfying $P \in (f_i)^{m_i}$ and $P \notin (f_i)^{m_i+1}$ in $\mca O_i$,
which exists by Nakayama's lemma. 
In other words, $m_i$ is the valuation of $P$ at $(f_i)$.

The notion of multiplicity $m_i$ is local, in the sense that it may be determined only by information in the local ring at a point on $C_i$.  
To see this, we invoke Serre's intersection formula (cf.\,\cite[Appendix A]{Hartshorne}): 
Let $X,Y$ be algebraic sets in $\C^n$ that intersect properly, that is, 
every irreducible component of $X\cap Y$ has codimension equal to $\operatorname{codim} X + \operatorname{codim} Y$, 
and let $W$ be an irreducible component of $X\cap Y$. 
Serre showed in \cite{Serre1965ALM} that we may define \emph{the intersection multiplicity} of $X$ and $Y$ along $W$ with the required properties of the intersection forms by putting 
\[\iota(X,Y;\, W) = \sum_{j\geq 0} (-1)^j \,  \textrm{length}_A( \Tor_j^A(A/\mathfrak a, A/ \mathfrak b)),\]
where $A$ denotes the local ring of $\C^n$ at an arbitrarily chosen point $w$ of $W$, and $\mathfrak a, \mathfrak b$ denote the ideals of $X,Y$ in $A$. 
If $\bullet_{(w)}$ denotes the localization at $w$ by slight abuse of notation, then we have $\mathfrak a={\ker}\left(\C[\C^n]_{(w)}\to \C[X]_{(w)}\right)$, and $\mathfrak b={\ker}\left(\C[\C^n]_{(w)}\to \C[Y]_{(w)}\right)$. 
By \cite[Appendix 1, Theorem 1.1]{Hartshorne}, the intersection multiplicity is additive and is equal to one if $X$ and $Y$ intersect transversally. 

Now, let us combine the settings of the first and second paragraphs. 
Suppose in addition that $P\in \C[X]$ is the restriction of an irreducible element $P\in \C[\C^n]$ with $Y=\{P=0\}\subset \C^n$. 
Then, by what we have discussed so far, 
the multiplicity $m_i$ of $P$ along $C_i$ on $X$ satisfies 
\[m_i=\iota(X,Y;C_i).\]
Moreover, 
one can see that $\Tor_j^A(A/\mathfrak a, A/ \mathfrak b)=0$ for $j\neq 0$ and  
$\Tor_0^A(A/\mathfrak a, A/ \mathfrak b)=A/\mathfrak a \otimes_A A/ \mathfrak b = A/(\mathfrak a + \mathfrak b)$, so we have 
\[
m_i = {\rm length}_A A/(\mathfrak a +\mathfrak b), 
\]
which has been more classically known. 

\begin{example} 
Consider the surface $X=\{z-x^2=0\}$ in $\C^3$. For $P=z$, with the above notations, we have $C_1=\{x=z=0\}$. Since
$\C[X]=\C[x,y,z]/(z-x^2)\cong \C[x,y]$, one finds $f_1=x$ and $\mca O_1=(\C[x,y,z]/(z-x^2))_{(x)}\cong \C[x,y]_{(x)}$. One obtains
$(P)=(z)=(x^2)=(x)^2$ in $\C[X]$, $P\in (x^2)$, $P\not\in (x^3)$ in $\mca O_1$, 
so $m_1=2$ and the divisor of $P$ is $D(P)=2 \,C_1$. 

In addition, let $Y=\{z=0\}$ and $w=(0,0,0)$. Then, we have 
$\C[Y]=\C[x,y,z]/(z)$, 
$A=\C[\C^3]_{(w)}=\C[x,y,z]_{(x,y,z)}$, and the ideals $\mf{a}$, $\mf{b}$ of $X$, $Y$ satisfy $\mf{a}+\mf{b}=(z,z-x^2)=(z,x^2)$. 
We obtain $A/(\mf{a}+\mf{b})\cong \C[x,y]_{(x,y)}/(x^2)$, 
the longest chain is $xA/(z,x^2) \subsetneq A/(z,x^2)$, 
and hence ${\rm length}_A A/(\mf{a}+\mf{b})=2$. 
\end{example} 

We prove that the multiplicity behaves well under restriction to codimension one subvarieties:
\begin{lemma}
\label{lem:intermult} 
Let $X,Y$ be subvarieties of $\C^n$ and $P\in \C[\C^n]$. Fix $f_i\in \C[X]$, $C_i=\{f_i=0\}\subset X$, and $D_X(P)=\sum_i m_i C_i$  as  above. 
Let $Z$ be a hypersurface in $X$ such that $Z$ and $C_i$'s intersect properly. 
Let $C_i\cap Z=\bigcup_j C_{ij}$ denote the unique irreducible decomposition for each $i$ and 
let $D_Z(P)=\sum_{ij} m_{ij} C_{ij}$ denote the divisor of $Z$ defined by the restriction $P\in \C[Z]$. 
Then each $(i,j)$ satisfies 
\[m_i\leq m_{ij}.\] 
\end{lemma}
\begin{proof} Let $g\in \C[X]$ be an irreducible element with $Z=\{g=0\}$ and $f_{ij}\in \C[Z]$ irreducible with $C_{ij}=\{f_{ij}=0\}$. 
Since $\C[Z]=\C[X]/(g)$, we have $f_i\equiv \prod_{j=1}^{s_i} f_{ij}^{e_{ij}}$ mod $(g)$ with $e_{ij}\in \Z_{> 0}$. 
Thus, if $P\in (f_i)^{m_i}$ in $O_i=\C[X]_{(f_i)}$, then we have $P\in (f_{ij})^{m_i e_{ij}}$ in $O_{ij}:=\C[Z]_{(f_{ij})}$. 
Since $m_{ij}$ is the integer satisfying $P \in (f_{ij}^{m_{ij}})$ and $P \notin (f_{ij}^{m_{ij}+1})$ in $O_{ij}$,
this implies that $m_i\leq m_{ij}$. 
\end{proof}




\begin{remark} 
The argument above naturally extends to any closed variety $X$ over $\C$ which is \emph{non-singular in codimension one}, that is, the set of singular points has codimension $\geq 2$, by 
choosing $w\in W$ which is not a singular point of $X$ (say, a generic point) and looking at open neighborhoods in the sense of Zariski topology. 
In this case, ${\rm Tor}_1^A$ may contribute. 
\end{remark} 

\section{The Whitehead link $W$}
\label{sec:Whitehead}
In this section, we thoroughly study the Whitehead link $W=W_1$ defined by the following diagram. 
We work over ${\rm SL}_2\C$. 
\begin{figure}[h]
\def\svgwidth{0.5\columnwidth}
\begingroup%
  \makeatletter%
  \providecommand\color[2][]{%
    \errmessage{(Inkscape) Color is used for the text in Inkscape, but the package 'color.sty' is not loaded}%
    \renewcommand\color[2][]{}%
  }%
  \providecommand\transparent[1]{%
    \errmessage{(Inkscape) Transparency is used (non-zero) for the text in Inkscape, but the package 'transparent.sty' is not loaded}%
    \renewcommand\transparent[1]{}%
  }%
  \providecommand\rotatebox[2]{#2}%
  \newcommand*\fsize{\dimexpr\f@size pt\relax}%
  \newcommand*\lineheight[1]{\fontsize{\fsize}{#1\fsize}\selectfont}%
  \ifx\svgwidth\undefined%
    \setlength{\unitlength}{352.32443718bp}%
    \ifx\svgscale\undefined%
      \relax%
    \else%
      \setlength{\unitlength}{\unitlength * \real{\svgscale}}%
    \fi%
  \else%
    \setlength{\unitlength}{\svgwidth}%
  \fi%
  \global\let\svgwidth\undefined%
  \global\let\svgscale\undefined%
  \makeatother%
  \begin{picture}(1,0.72519641)%
    \lineheight{1}%
    \setlength\tabcolsep{0pt}%
    \put(0,0){\includegraphics[width=\unitlength,page=1]{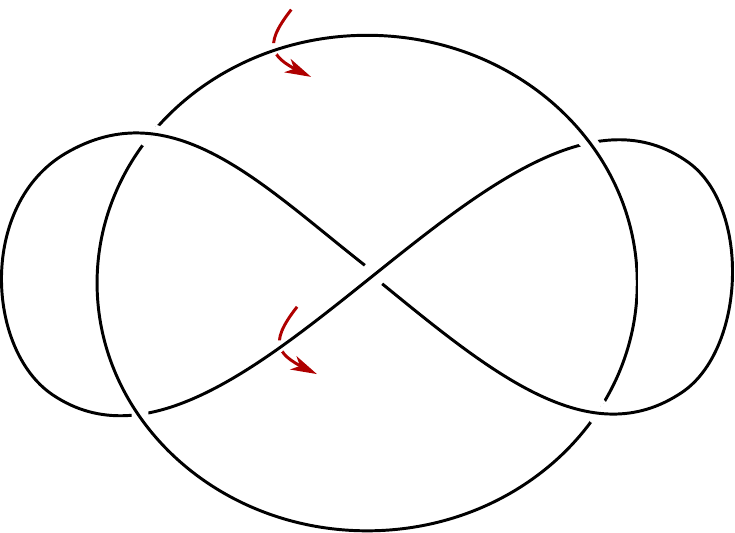}}%
    \put(0.39681741,0.71209146){\makebox(0,0)[lt]{\lineheight{1.25}\smash{\begin{tabular}[t]{l}$m$\end{tabular}}}}%
    \put(0.40475254,0.30719845){\makebox(0,0)[lt]{\lineheight{1.25}\smash{\begin{tabular}[t]{l}$\mu$\end{tabular}}}}%
  \end{picture}%
\endgroup%

\caption{\label{fig:Whitehead} 
A diagram of the Whitehead link, with meridians $m$ and $\mu$.}
\end{figure} 

We first recollect the calculation of character variety and the torsion function, following \cite{Tran2016JKTR.CV, Tran2018IJM,NguyenTran}. 
Afterward, we calculate the subvariety of non-acyclic representation and its multiplicity, 
as well as point out a geometric interpretation by using the $(-3,-3)$-filling, completing the proof of \Cref{thm.W}. 
We also revisit the study of twist knots $J(2,2l)$ to derive some corollaries. 

\subsection{The character variety of $W$} \label{subsec:char}
The group of $W$ has a presentation 
\[\pi_1(S^3-W) = \langle m,\mu \mid mw\mu^{-1}w^{-1}\rangle\]

with $w= \mu m \mu^{-1}m^{-1}\mu^{-1}m\mu=[\mu,m]\mu^{-1}m\mu$, 
where $m$ and $\mu$ denote the standard meridians in \cref{fig:Whitehead}. 
The relator $mw\mu^{-1}w^{-1}$ may be replaced by  
\[ [\mu,m][\mu^{-1},m] [\mu^{-1},m^{-1}][\mu,m^{-1}]. \]

Let $X(S^3-W)$ denote the ${\rm SL}_2\C$-character variety of $W$. 
For each $\sigma\in \pi_1(S^3-W)$, we associate the function $\Tr \sigma: X(S^3-W)\to \C;\, [\rho]\mapsto \Tr \rho(\sigma)$. 
Then 
\[x = \Tr m, \ \ y=\Tr \mu, \ \ z = \Tr m \mu\]
become coordinate functions. We put 
\begin{gather*}
f_1(x,y,z)=xyz^2-z(x^2+y^2+z^2)+xy+2z,\\
g(x,y,z)=x^2+y^2+z^2-xyz-4,
\end{gather*}
both of those are irreducible in $\C[x,y,z]$. 
The character variety of $W$ is described as 
\[X(S^3-W) = \{(x,y,z) \in \C^3 \mid f_1(x,y,z)g(x,y,z)=0\}.\] 
We have $f_1=g=0$ iff ($x=\pm 2$ and $y=\pm z$) or ($x=\pm z$ and $y=\pm 2$). %
The subvariety of reducible characters is $\{g=0\}$, 
and that of irreducible characters is 
\[X^{\rm irr}(S^3-W)=\{f_1=0\} \setminus 
\left(\{x=\pm 2, y=\pm z \} \cup \{y=\pm 2, x=\pm z \}\right).\]  

\subsection{The torsion function $\tau$ of $W$} \label{subsec:Tors}
Since the Reidemeister torsion of a representation $\rho$ is the value of the so-called twisted Alexander polynomial $\Delta_\rho(t_1,t_2)$ at $t_1=t_2=1$, 
\cite[Theorem 1]{NguyenTran} yields that the torsion function of $W$ is given by
\[
\label{eq:tors}
\tau(x,y,z)= 2(2+z-x-y).
\] 
\begin{proposition} \label{prop:acycl} 
The zero locus of the torsion function $\tau$ of $W$ on the variety $X^{\rm irr}(S^3-W)$ of irreducible ${\rm SL}_2\C$-characters is the 1-dimensional subvariety 
\[L:=\{x+y-1=0, z=-1\}\] 
and $\tau$ vanishes there with multiplicity two. 
The set of non-acyclic irreducible ${\rm SL}_2\C$-characters in $X^{\rm irr}(S^3-W)$ is described as  
\[L\cap X^{\rm irr}(S^3-W)= \{ x+y-1=0, \, z=-1 \}\setminus\{((2,-1,-1),(-1,2,-1)\}.\]
\end{proposition}

\begin{proof}
If we assume $\tau=2(2+z-x-y)=0$, then we have $z=x+y-2$, so we obtain $0=f_1=(x+y-1)^2(x-2)(y-2)$. 
If $x=2$, then we obtain $z=y$. If instead $y=2$, then we obtain $z=x$. 
In these two cases, these points are on $\{f_1=0\}\cap\{g=0\}$, so 
the factor $(x-2)(y-2)$ corresponds to reducible characters.  
If instead $x+y-1=0$, then $\tau=0$ yields $z=x+y-2=-1$. 
Since $\{x+y-1=0,z=-1\}\cap\{g=0\}=\{((2,-1,-1),(-1,2,-1)\}$, we obtain the assertion. 

Let $w$ be a generic point on $L$. Then, the local ring becomes
\[\left(\C[x,y,z]/(f,\tau)\right)_w=\big(\C[x,y,z]/\left((x+y-1)^2, z+1\right)\big)_w\]
with length two. Hence, by \Cref{subsec:mult}, the intersection multiplicity of $\{f_1=0\}$ and $\{\tau=0\}$ along $L$ is two. 
\end{proof}

\begin{figure}[h]
\includegraphics[width=100mm
]{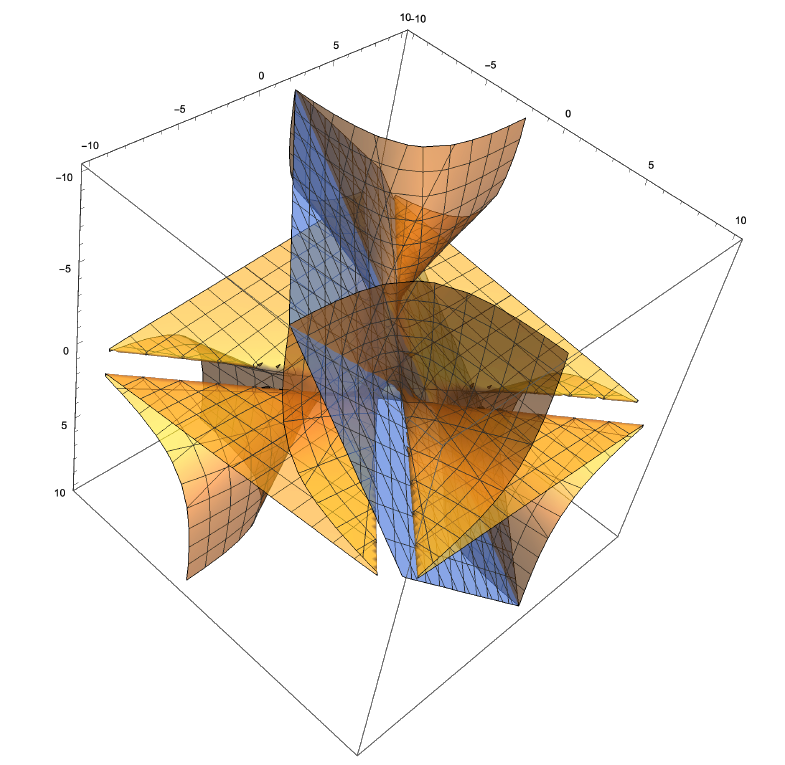}
\caption{$f(x,y,z)=0$, $\tau(x,y,z)=0$, and the line $L$}
\end{figure}

\subsection{The manifold $W(-3,-3)$}
\label{subsec:-3} 
We utilize the following well-known lemma. 
\begin{lemma} \label{lem.tr=-1}
Let $\F$ be any field. Then $A\in {\rm SL}_2\F$ satisfies $\Tr A=-1$ if and only if ${\rm ord}\,A=3$ holds.
\end{lemma}
\begin{proof} Note that $A^2-(\Tr A)A+I_2=0$ yields
$A^3-I_2
=(\Tr A+1)((\Tr A-1)A-I_2)$.

If $\Tr A=-1$, then we have $A\neq I_2$ and $A^3-I_2=O$, hence ${\rm ord}\,A=3$. 

Conversely, if ${\rm ord}\,A=3$, then we have $\Tr A+1=0$ or $(\Tr A-1)A-I_2=0$. 
Now suppose $(\Tr A-1)A-I_2=0$. 
Then, we have $\Tr A-1\neq 0$ and $A=\frac{1}{\Tr A-1}I_2$, 
so we may write $A=aI_2$. 
If ${\rm char}\,\F= 2$, then $\Tr A=2a=0\neq -1$, hence a contradiction. 
If instead ${\rm char}\,\F\neq 2$, then we have $a=\frac{a}{2a-1}$, $2a^2-2a=0$, $a=0,1$, hence again a contradiction. 
Thus, we have $\Tr A+1=0$. 
\end{proof}

\begin{proposition} 
\label{prop:M-3} 
Let $W(-3,-3)$ denote the result of $(-3,-3)$-filling along $W$. Regard its ${\rm SL}_2\C$-character variety $X(W(-3,-3))$ as a subvariety of $X(S^3-W)$. Then 
\[X^{\rm irr}(W(-3,-3))=L \cap X^{\rm irr}(S^3-W)\] 
holds. 
The fundamental group $\pi_1(W(-3,-3))$ is isomorphic to the triangle group 
\[\Gamma(3,3,\infty) \cong  \langle a,b \mid a^3 = b^3 = 1\rangle \cong \Z/3 \ast \Z/3.\]
\end{proposition}

\begin{proof}
Let $\rho$ be a non-acyclic irreducible representation of $\pi_1(S^3-W)$. 
Since $z = \Tr m\mu$, the equation $z=-1$ for the subvariety $L$ implies that $\rho(m\mu)$ has order 3. 
In addition, by using the ${\rm SL}_2\C$-trace relation $\Tr uv + \Tr u^{-1}v = \Tr u \Tr v$, we may further verify that 
$\Tr(m^2\mu)=-x-y$, so the equation $x+y-1=0$ implies that $\rho(m^2\mu)$ has order 3. 
Thus we have $L\cap X^{\rm irr}(S^3-W) \subset X^{\rm irr}(W(-3,-3))$. 

To prove the opposite direction, 
note that $(a,b) := (m\mu, m^2\mu)$ is also a generating system of the group $\pi_1(S^3-W)$. Namely, by $(m,\mu)=(ba^{-1}, ab^{-1}a)$, we obtain
\[\pi_1(S^3-W) = \langle a,b \mid  ba^{-3}bab^{-2}a^3b^{-1}a^{-1}b \rangle,\]
and we may rewrite $L= \{\Tr a=\Tr b=-1\}$. 
By \cite[Eq.(4) and Remark 1]{GuillouxWill}, we have the following peripheral basis for $\pi_1(W)$:
\begin{gather*}
\mu'=a^{-2}b = \mu^{-1}m^{-1}\mu^{-1} m \mu, \quad
 \lambda' = a^{-2}bab^{-2}ab,\\
m' = b^{-1}a = \mu^{-1}m^{-1}\mu, \quad
 \ell'=b^{-1}ab^{-1}aba^{-3}ba.
 \end{gather*}
%
Let $W(-3)=W(\bullet ,-3)$ denote the result of the $-3$-filling on $S^3-W$ along the second component of $W$. 
We obtain its group by adding the relation $m'^3=\ell'$, namely, we have 
\[\pi_1(W)/ \langle\langle m'^3\ell'^{-1} \rangle\rangle\cong \pi_1(W(-3)) = \langle a,b \mid a^3 = b^3 \rangle.\]
Note that the former relator becomes trivial there. 
By adding the relation $\mu'^3=\lambda'$, we obtain 
\begin{gather*} 
\pi_1(W(-3)) /\langle\langle \mu'^3\lambda'^{-1} \rangle\rangle \cong \pi_1(W(-3,-3))\\
= \langle a,b \mid a^3 = b^3 = 1\rangle \cong \Gamma(3,3,\infty).
\end{gather*} 

Now recall that the character variety of the free group $\langle a,b \rangle$ on two generators is isomorphic to $\C^3$, generated by the traces $\Tr a, \Tr b$ and $\Tr ab$. 
In $X^{\rm irr}(\Gamma(3,3,\infty))$, the traces of $a$ and $b$ are fixed to $-1$ by \cref{lem.tr=-1}, and the trace of $ab$ varies freely with $\Tr ab\neq 2$. 
Thus we have $X^{\rm irr}(W(-3,-3))=X^{\rm irr}(\Gamma(3,3,\infty)) \subset 
\{\Tr a=-1, \Tr b=-1, \Tr ab\neq 2\}=L\cap X^{\rm irr}(S^3-W)$. 

Combining these above, we obtain the equality in $X^{\rm irr}(S^3-W)$.  
\end{proof} 

\begin{remark}
\label{remk:components}
The closure $\ol{X^{\rm irr}}(W(-3))$ of the variety of irreducible characters of the $-3$-filling consists of two components. 
Note that any element there maps the central element $a^3=b^3$ to a central element $\pm I_2 \in {\rm SL}_2\C$. 
One of the components is the line $L$, which is the set of characters that factors through $\Gamma(3,3,\infty)$. 
Another component is described by $\{\Tr a = \Tr b = 1\}=\{z=1, x-y-1 = 0\}$ in $X(S^3-W)$. 
\end{remark}

\begin{proof}[Proof of \Cref{thm.W}] 
Propositions \ref{prop:acycl} and \ref{prop:M-3} yield the assertion.
\end{proof}

Finally we deduce \cref{cor:-3}:

\begin{proof}[Proof of \cref{cor:-3}] 
By \cite[Table 02]{MP} (see also a note right after Proposition 4 in \cite{GuillouxWill}), 
the space $W(-3,-3)$ is homeomorphic to the connected sum $L(3,1)\# L(3,1)$ of two lens spaces. 
Since the (non-)acyclicity persists under a filling away from the closed subset of meridional traces
$x=\Tr \rho(m')$ and $y=\Tr \rho(\mu')$ being 2 (\Cref{lem:MV}), 
\Cref{prop:M-3} yields that non-acyclic characters of $W(-3,-3)$ form a dense open subset of $L$ in $X^{\rm irr}(S^3-W)$. 
Since the acyclicity is an open property on $\ol{X^{\rm irr}}(S^3-W)$ (\cref{lem:open}), the non-acyclicity extends to the whole $L$. 
\end{proof}

We find \Cref{cor:-3} surprising: as we wrote before, acyclic representations are usually abundant, whereas it provides a simple example of a 3-manifold whose ${\rm SL}_2\C$-character variety is a curve that contains no acyclic irreducible representations. 

It is a Zariski open property on the character variety, and any hyperbolic $3$-manifold of finite volume possesses at least one acyclic representation, see \cite[Theorem 0.3 and Corollary 2.2]{MenalFerrerPorti2012}. In fact \cite[Theorem 0.3]{MenalFerrerPorti2012} is stated for only one cusp, but the same argument (applying Theorem 0.1) obviously works in the same way when $M$ has more cusps. For Seifert fibered $3$-manifolds, it is not difficult to see (cf.\,\cite{Kitano1994TJM}) that acyclic representations form a non-empty open subset in $X(M)$. 

\subsection{Revisiting twist knots}
\label{subsec:Twist} 
Let $r\in \Q\cup \{\infty\}$ and let $W(r)=W(r,\bullet)$ denote the result of the $r$-filling of $S^3-W$ along the first component of $W$. 
Then the variety $X^{\rm irr}(W(r))$ of irreducible ${\rm SL}_2\C$-characters of $W(r)$ may be seen as a subvariety of $X^{\rm irr}(S^3-W)$. 
Since the locus with meridional traces being two is contained in the subvariety of reducible characters of $S^3-W$, \cref{lem:MV} assures the following. 
\begin{lemma} \label{lem.capL} 
The set of non-acyclic irreducible characters of $W(r)$ is exactly the intersection 
$X^{\rm irr}(W(r))\cap X^{\rm irr}(W(-3,-3))$ in $X^{\rm irr}(S^3-W)$,  
where $X^{\rm irr}(W(-3,-3))$ $=L\cap X^{\rm irr}(S^3-W)$ is the line consisting of non-acyclic characters of $S^3-W$. 
\end{lemma} 

\begin{proof}[Proof of \Cref{cor:mult2Wr}] 
By \cref{remk:components}, the multiplicity of $D(\tau)$ on $X^{\rm irr}(W(r))$ is defined unless $r=-3$.  
\cref{lem.capL} and \cref{lem:intermult} assure that the multiplicity is at least two. 
\end{proof}

Especially if $r=-1/l$ with $l\in \Z$, then $W(-1/l)$ becomes the exterior $S^3-J(2,2l)$ of the twist knot $J(2,2l)$. 
The property ``multiplicity exactly two'' proved in \cite[Theorem A]{TTU} may be interpreted as the transversality of the intersection. 
Let us further observe the group of $W(-1/l)$. 
We put $\upsilon = m^{-1}\mu m$, which is a meridian of the second component of $W$. 
Then the surgery longitude becomes $\lambda = \upsilon^{-1} \mu \upsilon \mu^{-1}$, and we obtain the presentation 
\[ \pi_1(M_{-1/l}) = \langle m,\upsilon \mid \lambda^l m = \upsilon \lambda^l \rangle \]
of the twist knot exterior, and the coordinate 
$\alpha = \Tr \mu = \Tr \upsilon, \, \beta = \Tr \mu \upsilon$
coincides with the one used in \cite{TTU}. 
We previously proved that an irreducible representation $\rho$ of $W(-1/l)$ factors through the $-3$-filling iff $\Tr \rho(m) = \Tr \rho(m\upsilon)$ holds 
\cite[Theorem C]{TTU}. 
In our coordinate, this equation becomes $xyz-y^2-z^2-x+2 = 0$. 
By substituting $z=-1$, we obtain $-(y+1)(x+y-1)=0$. 
Since $y+1\neq 0$ on $X^{\rm irr}(S^3-W)$, we see that 
an irreducible character of $W(-1/l)$ is non-acyclic if and only if it factors through the $-3$-filling and $z=-1$ holds, recovering \cite[Theorem D]{TTU}. 

We may further prove \cref{cor.SU2} on $W(-1/l)=S^3-J(2,2l)$: 

\begin{proof}[Proof of \cref{cor.SU2}] 
A non-acyclic representation $\rho$ of $\pi_1(W(-1/l))$ satisfies 
\[ \rho(a^3) = \rho(b^3) = \Id, \quad \rho(\mu') = \rho(\lambda')^l.\]
Hence, by $\rho(a^{-2}) = \rho(a), \rho(b^{-2}) = \rho(b)$, the second relation becomes
\[\rho(ab) = \rho(ab)^{3l}.\]
So the representation $\rho$ factors through a representation $\Gamma(3,3,|3l-1|)\to {\rm SL}_2\C$.
\end{proof}

\section{Odd-twisted Whitehead links} 
We now consider the family of odd-twisted Whitehead links. 
We continue to work on ${\rm SL}_2\C$. 
We prove that the character variety is smooth in \cref{subsec:smooth}. In \cref{subsec:vanish} we compute the vanishing multiplicity of the Reidemeister torsion on the character variety of the twisted Whitehead link, and prove \cref{theo:TW}. Then in \cref{subsec:double} we prove \cref{coro:doubletwist}. 

\subsection{Twisted Whitehead links $W_k$}
\label{sec:TW} 
For each $k\in \Z$, the $k$-twisted Whitehead link $W_k$ is defined by the following diagram. 
We have that $W_{-1}$ is the trivial two-component link, $W_0$ is the torus link $T(2,4)$, and $W_1$ is the Whitehead link $W$. 
For each $k\in \Z$, $W_k$ and $W_{-k-2}$ are the mirror images of each other. 

\begin{figure}[h]
\def\svgwidth{0.5\columnwidth}
\begingroup%
  \makeatletter%
  \providecommand\color[2][]{%
    \errmessage{(Inkscape) Color is used for the text in Inkscape, but the package 'color.sty' is not loaded}%
    \renewcommand\color[2][]{}%
  }%
  \providecommand\transparent[1]{%
    \errmessage{(Inkscape) Transparency is used (non-zero) for the text in Inkscape, but the package 'transparent.sty' is not loaded}%
    \renewcommand\transparent[1]{}%
  }%
  \providecommand\rotatebox[2]{#2}%
  \newcommand*\fsize{\dimexpr\f@size pt\relax}%
  \newcommand*\lineheight[1]{\fontsize{\fsize}{#1\fsize}\selectfont}%
  \ifx\svgwidth\undefined%
    \setlength{\unitlength}{352.32443718bp}%
    \ifx\svgscale\undefined%
      \relax%
    \else%
      \setlength{\unitlength}{\unitlength * \real{\svgscale}}%
    \fi%
  \else%
    \setlength{\unitlength}{\svgwidth}%
  \fi%
  \global\let\svgwidth\undefined%
  \global\let\svgscale\undefined%
  \makeatother%
  \begin{picture}(1,0.72519641)%
    \lineheight{1}%
    \setlength\tabcolsep{0pt}%
    \put(0,0){\includegraphics[width=\unitlength,page=1]{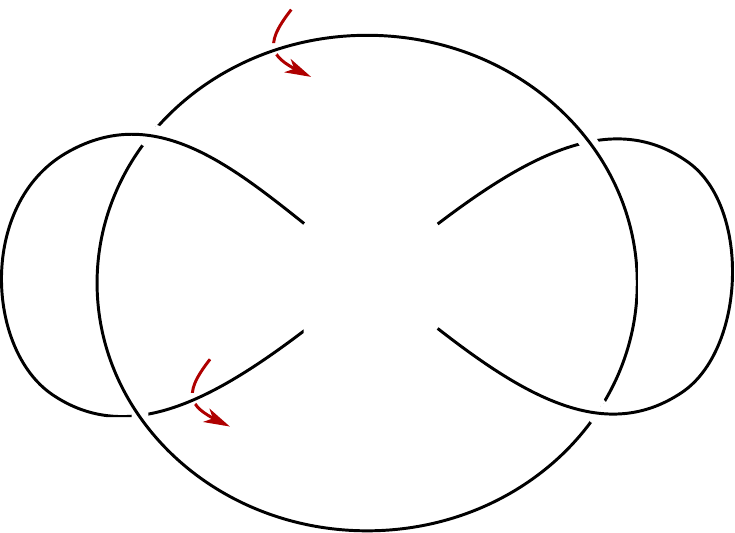}}%
    \put(0.39681741,0.71209146){\makebox(0,0)[lt]{\lineheight{1.25}\smash{\begin{tabular}[t]{l}$m$\end{tabular}}}}%
    \put(0.28615242,0.2357343){\makebox(0,0)[lt]{\lineheight{1.25}\smash{\begin{tabular}[t]{l}$\mu$\end{tabular}}}}%
    \put(0,0){\includegraphics[width=\unitlength,page=2]{Twisted.pdf}}%
    \put(0.45334371,0.38607616){\makebox(0,0)[lt]{\lineheight{1.25}\smash{\begin{tabular}[t]{l}$k$\\half\\twists\end{tabular}}}}%
  \end{picture}%
\endgroup%

\caption{\label{fig:TwistedW} A diagram of twisted Whitehead link $W_k$ with meridians $m$ and $\mu$.} 
\end{figure} 
Its group has a presentation 
\[ \pi_1(S^3-W_k)= \langle m, \mu \mid m\omega = \omega m \rangle, \] 
where 
\[ \omega = \begin{cases}
(\mu m \mu^{-1} m^{-1})^n m (m^{-1} \mu^{-1} m \mu)^n &\text{ if } k = 2n-1\\
(\mu m \mu^{-1} m^{-1})^n \mu m \mu (m^{-1} \mu^{-1} m \mu)^n &\text{ if } k = 2n
\end{cases}
\] 
with $n\in \Z$, 
and $m,\mu$ denote the meridians displayed in \cref{fig:TwistedW}. 

The third author initially calculated the character variety of $W_k$ in \cite{Tran2016JKTR.CV} and, afterward, he did it again in a different way, as well as detecting the geometric component in \cite{Tran2018IJM}.  
The torsion function was studied in detail in \cite{NguyenTran}. 
Although in the literature they assumed $k\geq 0$, the recurrence formula applies to any $k\in \Z$, so the results persist for any $k\in \Z$. 

We define the Chebychev polynomials $S_k(v)\in \Z[v]$ $(k\in\Z)$ of the second kind by 
\[S_0(v) = 1, \, S_1(v) = v, \, S_{k+2}(v) = vS_{k+1}(v) - S_k(v).\] 
We have $S_{k-1}(v)=-S_{-k-1}(v)$ and $S_{k-1}(\pm 2)=(\pm 1)^k k$. 
The following lemma is easy to see. 
\begin{lemma} \label{lem:Cheb} 
For any $k\in \Z$, if $v= w + w^{-1}$ with $w\in \C$, then $S_{k-1}(v) = \frac {w^k - w^{-k}}{w-w^{-1}}$. 
If $k\neq 0$, then we have 
\[S_{k-1}(v)=\prod_{0<j<|k|}(v-2\cos \frac{j\pi}{k}).\] 
\end{lemma}

\subsection{The character variety of $W_{2n-1}$} \label{sec:tW} 

In the following, we assume $0\neq n \in \Z$. 
We again use the coordinate functions \[x = \Tr m,\ y=\Tr \mu,\ z = \Tr m \mu.\]  
In addition, we put \[v = \Tr (\mu m \mu^{-1} m^{-1}),\]
so that we have  
$g(x,y,z)=x^2 + y^2 + z^2 - xyz - 4=v-2$ and $S_k(v)\in \Z[x,y,z]$. 
Put 
\[f_n(x,y,z)=\ xy S_{n-1}(v) - (xy-z) S_{n-2}(v) -  z S_n(v) \in \Z[x,y,z].\]
Then the character variety of $W_{2n-1}$ is described as 
\[X(S^3-W_{2n-1})=\{(v+2)f_n(x,y,z)S_{n-1}(v)=0\}.\]
The first term $v+2=g(x,y,z)$ corresponds to reducible characters. 
The middle term $f_n(x,y,z)$ corresponds to the geometric component, that is, the class of holonomy representations is there. 
The last term satisfies $S_{n-1}(v)=\prod_{0<j<|n|}(v-2\cos \frac{j\pi}{n})$, so it corresponds to $|n|-1$ components. 

\subsection{Smoothness}
\label{subsec:smooth}
In this subsection, we prove \cref{prop:smooth} asserting that 
every component of the character variety $X(S^3-W_{2n-1})$ of $W_{2n-1}$ is smooth.

\subsubsection{The geometric component}
We show that the following system has no solution: 
\[
\left\{ 
\begin{alignedat}{4}
f\ &= xy S_{n-1}(v) - (xy-z) S_{n-2}(v) -  z S_n(v) &=0\\
f_x &= (xy S'_{n-1}(v) - (xy-z) S'_{n-2}(v) -  z S'_n(v))(2x-yz) + y (S_{n-1}(v) - S_{n-2}(v))\ &=0\\
f_y &= (xy S'_{n-1}(v) - (xy-z) S'_{n-2}(v) -  z S'_n(v))(2y-xz) + x (S_{n-1}(v) - S_{n-2}(v))\ &=0\\
f_z &= (xy S'_{n-1}(v) - (xy-z) S'_{n-2}(v) -  z S'_n(v))(2z-xy) -  (S_{n}(v) - S_{n-2}(v)) &=0\\
\end{alignedat}
\right.
\]
where $f := f_n$ is the polynomial of the geometric component, $f_x, f_y, f_z$ stand for the partial derivatives of $f$, and $S_k'(v)$ denotes the derivative of $S_k$ in the variable $v$. 
The proof splits into a disjunction of Cases 1--4.

\medbreak

\paragraph{\bf Case 1} Suppose $x^2-y^2 \neq 0$. 
Then $f_x=f_y=0$ yields 
\[xy S'_{n-1}(v) - (xy-z) S'_{n-2}(v) -  z S'_n(v)=\frac{xf_x-yf_y}{2(x^2-y^2)}=0.\]
Since we have $(x,y)\neq (0,0)$, at least one of $f_x=0$ or $f_y=0$ yields 
\[S_{n-1}(v) - S_{n-2}(v)=0.\]
In addition, by $f_z=0$, we have 
\[S_{n}(v) - S_{n-2}(v)=0.\]
This never holds by the following \Cref{lem.Snn-1n-2}. 

\begin{lemma} \label{lem.Snn-1n-2}
No $v\in \C$ satisfies the condition 
\[S_n(v) = S_{n-1}(v) = S_{n-2}(v).\]
\end{lemma} 
\begin{proof}
Suppose that these equations hold.  
Then, the recurrence formula yields 
\[(2-v) S_{n-1}(v) = 0. \]
Since we have $S_k(2)= k+1$ for every $k\in \Z$, if $v=2$, then we obtain a contradiction. 
If instead $S_{n-1}(v) = 0$, then we have $S_n(v) = S_{n-2}(v) =0$ as well, contradicting to \Cref{lem:Cheb}. 
This completes the proof. 
\end{proof}

\medbreak

\paragraph{\bf Case 2} Suppose $(x,y) = (0,0)$. Then the system becomes 
\[
\left\{ 
\begin{alignedat}{4}
z(S_{n-2}(v) - S_n(v))&=0\\
x&=0\\
y&=0\\
2z^2 (S'_{n-2}(v) - S'_n(v)) - (S_{n-2}(v) - S_n(v)) &=0 
\end{alignedat}
\right.
\]
If $z=0$, then we have $v = x^2+y^2+z^2-xyz-2=-2$. 
Since $S_k(-2) = (-1)^k (k+1)$ for every $k\in \Z$, 
we obtain $S_n(-2) - S_{n-2}(-2) =(-1)^n2 \neq 0$, hence a contradiction. 

If instead $z\neq 0$ and $S_{n-2}(v) - S_n(v)=0$, 
then we have $S'_{n-2}(v) - S'_n(v)=0$, 
which contradicts to the following \cref{lem:Sep} asserting that every root of $S_{n-2}(v) - S_n(v)$ has multiplicity one. 

\begin{lemma}
\label{lem:Sep}
If $0\neq k\in \Z$, then we have 
\[S_{k}(v) - S_{k-2}(v) = \prod_{0\leq j <|k|} \left( v -  2\cos \frac{(2j+1)\pi}{2|k|} \right).\]
\end{lemma}

\begin{proof} 
If $0<k$, then \cref{lem:Cheb} assures that $S_{k} - S_{k-2}$ is a polynomial of degree $n$ in $v$ with leading coefficient 1. 
If we put $v = 2\cos \frac{(2j+1)\pi}{2k}$ with $0 \leq j < k$, then by
\[S_k(v) = \frac{\sin \frac{(k+1)(2j+1)\pi}{2k}}{\sin \frac{(2j+1)\pi}{2k}},\]
we obtain   
\[
S_{k}(v) - S_{k-2}(v) = \frac{\sin \frac{(k+1)(2j+1)\pi}{2k} - \sin \frac{(k-1)(2j+1)\pi}{2k} }{\sin \frac{(2j+1)\pi}{2k}} =0.
\]

By $S_{k-1}(v)=-S_{-k-1}(v)$, we have 
$S_{k}(v) - S_{k-2}(v) = -S_{-k-2} +S_{-k}(v)=S_{-k}(v)-S_{-k-2}$, 
so we obtain the cases with $k<0$. 
(We remark that $S_0-S_{-2}=1-(-1)=2\neq 1$.) 
\end{proof} 

\medbreak

\paragraph{\bf Case 3} Suppose $x=y \neq 0$. Then the system becomes
\[
\left\{ 
\begin{alignedat}{4}
f\ &= x^2 (S_{n-1}(v) -S_{n-2}(v))  -  z (S_n(v)-S_{n-2}(v))\ &=0\\
f_x &=h(x,z,v)x(2-z) + x (S_{n-1}(v) - S_{n-2}(v))\ &=0\\
f_z &=h(x,z,v)(2z-x^2) -  (S_{n}(v) - S_{n-2}(v)) &=0\\
\end{alignedat}
\right.
\]
with $h(x,z,v)= x^2 S'_{n-1}(v) - (x^2-z) S'_{n-2}(v) -  z S'_n(v)$. 

We have $S_{n}(v)-S_{n-2}(v)\neq 0$. Indeed, if $S_{n}(v)-S_{n-2}(v)=0$, then we have $0=f/x^2=S_n(v)-S_{n-2}(v)$, contradicting to \Cref{lem.Snn-1n-2}. 

Hence, $f_z=0$ yields $h(x,z,v)(2z-x^2)\neq 0$. 
By 
\[0=f-xf_x-zf_z=-2h(x,z,v)(x^2-x^2z+z^2),\]
we obtain $x^2-x^2z+z^2=0$, so $v=2x^2+z^2-x^2z-2=x^2-2$. 

We have $v\neq \pm 2$. Indeed, $x\neq 0$ implies $v\neq -2$. 
In addition, if $v=2$, then by $S_{k-1}(2)=k$, we have $0=f=x^2-2z$ and $0=f_z=0-2$, hence a contradiction. 

Since $x\neq 0$, we may write $x = a + a^{-1}$ with $a\in \C$. 
Note that $v = x^2-2 = a^2 + a^{-2}$ and $a^2 \neq \pm1$. 
By $S_{k-1}(v) = \frac{a^{2k} - a^{-2k}}{a^2-a^{-2}}$, we have 
\[z = x^2 \frac{S_{n-1}(v) - S_{n-2}(v)}{S_{n}(v) - S_{n-2}(v)} = x^2 \frac{a^{4n} + a^2}{(a^2+1)(a^{4n} +1)} = x \frac{a^{4n-1} + a}{a^{4n} +1},\] 
\begin{eqnarray*}
z^2 - x^2 z + x^2 &=& x^2 \left( \frac{(a^{4n-1} + a)^2}{(a^{4n} +1)^2} - (a+a^{-1}) \frac{a^{4n-1} + a}{a^{4n} +1}+ 1\right) \\
&=& x^2 \left( \frac{a^{4n-1} + a}{a^{4n} +1} - a^{-1}  \right)  \left( \frac{a^{4n-1} + a}{a^{4n} +1} - a\right) \\
&=& - x^2 \, \frac{a-a^{-1}}{a^{4n} +1} \, \frac{(a-a^{-1})a^{4n}}{a^{4n} +1} \\
&\not=& 0, 
\end{eqnarray*}
hence a contradiction. 

\medbreak

\paragraph{\bf Case 4} Suppose $x=- y \neq 0$. Then the system becomes 
\[
\left\{ 
\begin{alignedat}{4}
f\ &= -x^2 (S_{n-1}(v) -S_{n-2}(v))  -  z (S_n(v)-S_{n-2}(v))\ &=0\\
f_x &=h(x,z,v)x(2+z) + x (S_{n-1}(v) - S_{n-2}(v))\ &=0\\
f_z &=h(x,z,v)(2z+x^2) -  (S_{n}(v) - S_{n-2}(v)) &=0\\
\end{alignedat}
\right.
\]
with $h(x,z,v)= -x^2 S'_{n-1}(v) + (x^2z) S'_{n-2}(v) -  z S'_n(v)$. 
A similar argument to Case 3 yields that $h(x,z,v)\neq 0$, $x^2+x^2z+z^2=0$, 
and $v=x^2-2\neq \pm 2$. 
We may write $x=a+a^{-1}$. We have 
\[z = - x^2 \frac{S_{n-1}(v) - S_{n-2}(v)}{S_{n}(v) - S_{n-2}(v)} = - x \frac{a^{4n-1} + a}{a^{4n} +1},\] 
\[z^2 + x^2 z + x^2 
= x^2 \left( \frac{(a^{4n-1} + a)^2}{(a^{4n} +1)^2} - (a+a^{-1}) \frac{a^{4n-1} + a}{a^{4n} +1}+ 1\right) \neq 0,
\]
hence a contradiction. 
Thus, the system $f = f_x = f_y = f_z =0$ has no solutions.

\subsubsection{The other components}
There are $|n|-1$ other components given by $S_{n-1}(v)=\prod_{0<j<|n|}(v-2\cos \frac{j\pi}{n})=0$. 
Recall that $v=x^2+y^2+z^2-xyz-2$. 
For each $j$, we consider the system
\[
\left\{ 
\begin{alignedat}{4}
v-2\cos \frac{j\pi}{n}&=0\\
2x-yz&=0\\
2y-xz &=0\\
2z-xy &=0.\\
\end{alignedat}
\right.
\] 
The solution of the last three equations is 
$x=y=z=0$ or $x,y,z\in \{\pm 2\}$ with $xyz=8$, 
so we have $v=\mp 2\neq 2\cos \frac{j\pi}{n}$. 
Thus, there is no solution. 

This completes the proof of \cref{prop:smooth}. 

\subsection{The vanishing multiplicities of the torsion functions} 
\label{subsec:vanish}

In this subsection, we compute the vanishing multiplicity of the torsion functions of $W_{2n-1}$'s 
and complete the proof of \cref{theo:TW}.  

We put $P_k(v) = (S_{k+1}(v) - S_k(v)-1)/(v-2) \in \Z[v]$. 
By \cite[Corollary 2]{NguyenTran}, 
the torsion function of $W_{2n-1}$ becomes 
\[ \tau_n(x,y,z)=(2-x-y+z)S_{n-1}(v) + (4-2x-2y+xy)P_{n-2}(v). \]
(Note that $\tau_{M_{2n-1},p}^\rho$ in \cite[Corollary 2]{NguyenTran} is the torsion of a surgered manifold, so we have 
$\tau_{M_{2n-1},p}^\rho = \tau_n(x,y,z) \frac{2}{x-2}$.)  

We consider the elements of $\C[x,y,z,w]$ defined by 
\begin{gather*}
f_n(x,y,z,v)=xy S_{n-1}(v) - (xy-z) S_{n-2}(v) -  z S_n(v),\\ 
\tau_n(x,y,z,v)=(2-x-y+z)S_{n-1}(v) + (4-2x-2y+xy)P_{n-2}(v). 
\end{gather*} 
In order to calculate the multiplicity of the divisor of the torsion function on $X^{\rm irr}(S^3-W_{2n-1})$, 
it suffices to study the ideal of $\C[x,y,z,v]$ generated by 
\[g(x,y,z)-2-v,\ f_n(x,y,z,v) S_{n-1}(v),\ \tau_n(x,y,z,v).\]

\subsubsection{The non-geometric components} 
We first consider the components defined by $S_{n-1}(v)=\prod_{0<j<|n|}(v-2\cos\frac{j\pi}{n})=0$ and prove the following.
\begin{proposition} \label{prop.Sv} 
On the subvariety of $X^{\rm irr}(S^3-W_{2n-1})$ defined by $v-2\cos\frac{j\pi}{n}=g(x,y,z)-2-2\cos\frac{j\pi}{n}=0$ with $0<j<|n|$, 

{\rm (1)} If $j$ is even, then $\tau_n=0$ holds identically. 

{\rm (2)} If $j$ is odd, then the set of $\tau_n=0$ is the union of 4 lines, that is, 
\[
\{x=2,\  y-z= \pm 2\sqrt{-1} \sin \frac{j\pi}{2n}\}\cup 
\{y=2,\  x-z= \pm 2\sqrt{-1} \sin \frac{j\pi}{2n}\}, \]
and $\tau_n$ vanishes there with multiplicity 1. 
\end{proposition} 

\begin{proof} 
Assume $v=2\cos \frac{j\pi}{n}$. Then by $S_{n-1}(v)=0$, we have $\tau_n=(2-x)(2-y)P_{n-2}(v)$. 
In addition, 
by $S_{n-2}(v)=\frac{\sin\frac{j(n-1)\pi}{n}}{\sin\frac{k\pi}{n}}=(-1)^{j-1}$, 
we have $P_{n-2}(v)=\frac{-(-1)^{j-1}-1}{v-2}$. 

(1) If $j$ is even, then by $P_{n-2}(v)=0$, we have $\tau_n=0$ identically. 

(2) If $j$ is odd, then we have $\tau_n= (2-x)(2-y)\frac{-2}{v-2}$ with $v-2\neq 0$. 
If $x=2$, then we have  $g-v-2=(y-z)^2-(v-2)=(y-z)^2-4(\sin \frac{j\pi}{2n})^2
=(y-z+2\sqrt{-1}\sin \frac{j\pi}{2n})(y-z-2\sqrt{-1}\sin \frac{j\pi}{2n})
$. 
If $y=2$, then we have similar. 
Since the vanishing multiplicity of $\tau_n$ may be seen as the length of the localization of the ring 
\[\C[x,y,z,v]/(g-v-2, v-2\cos \frac{j\pi}{n}, (2-x)(2-y))\] 
at a generic point on  $v-2\cos\frac{j\pi}{n}=0$, 
except for the intersection of components of $\{\tau=0\}$, we have multiplicity 1. 
\end{proof}

\subsubsection{The geometric component} 
\label{subsec:case2} 
We next prove the following. 

\begin{proposition} \label{prop.gc} 
On the subvariety of $X^{\rm irr}(S^3-W_{2n-1})$ defined by 
\[xy S_{n-1}(v) - (xy-z) S_{n-2}(v) -  z S_n(v)=0\]
and $S_{n-1}(v) \neq 0$, $\tau_n$ vanishes with multiplicity 2. 
\end{proposition}

\begin{proof} Consider the ideal of $\C[x,y,z,v]$ generated by 3 elements 
\begin{gather*} q(x,y,z,v)=\ x^2 + y^2 + z^2 - xyz - 2 - v,  \\
f_n(x,y,z,v) = xy S_{n-1}(v) - (xy-z) S_{n-2}(v) -  z S_n(v),\\
\tau_n(x,y,z,v) = (2-x-y+z)S_{n-1}(v) + (4-2x-2y+xy)P_{n-2}(v).
\end{gather*} 
Choose any prime ideal $\mf{p}$ underlying the primary decomposition of this ideal and consider the localization of $\C[x,y,z,v]_{\mf{p}}$.
Then the polynomials $S_{n-1}(v),$ $v-2,$ $T_n(v)-2,$ and $2 T_n(v) -v +2$ are inveretlible there. 
Let $\doteq$ denote the equality up to multiplication by units in each ring. 

\medbreak

First, in the local ring $\C[x,y,z,v]_{\mf{p}}$, we have 
\[ \tau_n(x,y,z,v)\doteq  (x + y -2) \left(1+ \frac{2P_{n-2}(v)}{S_{n-1}(v)}\right)- xy\frac{P_{n-2}(v)}{S_{n-1}(v)} - z.\] 
Hence, in $\C[x,y,z,v]_\mf{p}/(\tau_n)$, we have 
\begin{align*} f_n=&\ xy (S_{n-1}(v)  - S_{n-2}(v)) -  z (S_n(v) - S_{n-2}(v) )\\
=&\ xy (S_{n-1}(v)  - S_{n-2}(v)) -
 \left( (x + y -2)(1+ \frac{2P_{n-2}(v)}{S_{n-1}(v)})- xy\frac{P_{n-2}(v)}{S_{n-1}(v)} \right)(S_n(v) - S_{n-2}(v) )\\
= &\ xy \left( S_{n-1}(v)  - S_{n-2}(v) + \frac{P_{n-2}(v)}{S_{n-1}(v)}  (S_n(v) - S_{n-2}(v) )\right) \\ 
 & -
(x + y -2)\left(1+ \frac{2P_{n-2}(v)}{S_{n-1}(v)}\right) (S_n(v) - S_{n-2}(v) ).
\end{align*}
Write $v = a+ a^{-1}$. Then $S_k(v) = (a^{k+1} - a^{-k-1})/(a-a^{-1})$. By a direct calculation, we further obtain 
\begin{align*} f_n=&\ (2a^n + 2 a^{-n} - a - a^{-1} +2) xy - (a^n + a^{-n})(2+a+a^{-1})(x+y-2)\\
\dot{=}&\ xy- \frac{(a^n + a^{-n})(2+a+a^{-1}) }{2a^n + 2 a^{-n} - a - a^{-1} +2} (x+y-2).
\end{align*}
In $\C[x,y,z,v]_\mf{p}$ modulo the right hand side, we further obtain 
\[\tau_n(x,y,z,v)\doteq \frac{a^n + a^{-n}+a^{n-1} + a^{1-n} }{2a^n + 2 a^{-n} - a - a^{-1} +2} (x+y-2) -z.\] 

Therefore, in $\C[x,y,z,v]_\mf{p}/(f_n,\tau_n)$, we have 
\begin{align*}
q&=x^2 + y^2 + z^2 - xyz - 2 - v\\ 
&=(x+y)^2 + z^2 -2-v- xy (z+2)\\
&=(x+y)^2 + \left(\frac{a^n + a^{-n}+a^{n-1} + a^{1-n} }{2a^n + 2 a^{-n} - a - a^{-1} +2}\right)^2 (x+y-2)^2 -2-(a+a^{-1})\\
&\phantom{=} - \frac{(a^n + a^{-n})(2+a+a^{-1}) }{2a^n + 2 a^{-n} - a - a^{-1} +2} (x+y-2) \left(\frac{a^n + a^{-n}+a^{n-1} + a^{1-n} }{2a^n + 2 a^{-n} - a - a^{-1} +2} (x+y-2)+2 \right) \\
&=(a+a^{-1}-2) \left( x + y - \frac{a+a^{-1} +2}{a^n+a^{-n}+2}\right)^2\\
&=(v-2)\left( x + y - \frac{v+2}{T_n(v)+2}\right)^2\\ 
&\doteq \left( x + y - \frac{v+2}{T_n(v)+2}\right)^2. 
\end{align*}
Here $T_k(v)$'s are the Chebyshev polynomials of the first kind defined by $T_0(v) =2$, $T_1(v) =v$ and $T_k(v) = v T_{k-1}(v) - T_{k-2}(v)$ for all $k \in \Z$. We have  $T_k(a+a^{-1})=a^k + a^{-k}$. 

Thus the ideal $(g,f_n,\tau_n)$ of $\C[x,y,z,v]_\mf{p}$ is generated by 3 polynomials  
\begin{gather*}
z - \frac{T_n(v) + T_{n-1}(v) }{2T_n(v) - v+2} (x+y-2), \\
xy - \frac{T_n(v)(2+v) }{2T_n(v) - v+2} (x+y-2), \\
\left(x + y - \frac{v+2}{T_n(v)+2}\right)^2
\end{gather*} 
and we have an isomorphism 
\[\C[x,y,z,v]_\mf{p}/(q,f_n,\tau_n)\cong \C[x,y,v]_\mf{p}/(g,h^2)\]
where we put 
\[g= \frac{T_n(v)(2+v) }{2T_n(v) - v+2} (x+y-2)-xy, \quad h = x + y - \frac{v+2}{T_n(v)+2}.\]

Since the multiplicity is additive \cite[Appendix A, Theorem 1.1]{Hartshorne}, 
it suffices to show that the varieties of $g$ and $h$ intersect generically transversely. 
For this, note that $T_n'(v) = n S_{n-1}(v)$. 
Up to multiplication by units in the local ring $\C[x,y,v]_\mf{p}/(g,h)$, the Jacobian matrix becomes 
\[ \mathcal J(g,h)= \bma (2T_n(v) - v+2)y - T_n(v)(2+v) &T_n(v)+2 \\
 (2T_n(v) - v+2)x - T_n(v)(2+v) & T_n(v)+2 \\
 \partial_v g & \partial_v h\ema\] 
and the first $2\times 2$ minor is
$M = (T_n(v) +2) (x-y)(2T_n(v)-v+2)$.
This term is non-zero apart from a strict subvariety of the intersection $X_g \cap X_h$ of the varieties of $g$ and $h$, so the intersection is generically transversal. This completes the proof. 
\end{proof}

\begin{proof}[Proof of \cref{theo:TW}] 
By Propositions \ref{prop.Sv} and \ref{prop.gc}, the torsion vanishes with multiplicity two on the variety of irreducible representations of the twisted Whitehead link. 
\end{proof} 

\subsection{Double twist knots $J(2n,2l)$}
\label{subsec:double}
In this section, we consider $(-1/l)$-filling on the twisted Whitehead link $W_{2n-1}$. The resulting knots are genus one two-bridge knots $J(2n,2l)$. 
We deduce from \cref{theo:TW} that the torsion vanishes with multiplicity at least two on their character variety, proving \cref{coro:doubletwist}.

\begin{proof}[Proof of \cref{coro:doubletwist}] 

Suppose $\rho: \pi_1(S^3 - W_{2n-1})  \to {\rm SL}_2\C$ is a nonabelian representation of the form
\[
\rho(a) = \left[ \begin{array}{cc}
s_1 & 1 \\
0 & s_1^{-1} \end{array} \right] \quad \text{and} \quad 
\rho(b) = \left[ \begin{array}{cc}
s_2 & 0 \\
u & s_2^{-1} \end{array} \right]
\]
where 
\[
\big( xy S_{n-1}(v) -  (xy-z)  S_{n-2}(v) - z S_n(v) \big) S_{n-1}(v) = 0.
\]
Here $x = s_1 + s_1^{-1}$, $y = s_2 + s_2^{-1} $, $z = s_1 s_2 + s_1^{-1} s_2^{-1} + u$ and $v = x^2 + y^2 + z^2 - xyz-2$ are the traces of the images of $a$, $b$, $ab$ and $bab^{-1}a^{-1}$ respectively.

We first claim that if $\rho$ satisfies $S_{n-1}(v)=0$, then it cannot be extended to a representation $\pi_1(S^3 - J(2n,2l)) \to {\rm SL}_2\C$. Indeed, let $\lambda_a$ be the canonical longitude corresponding to the meridian $a$ of $W_{2n-1}$. By \cite{Tran2018IJM}, we have
\[ \rho(\lambda_a) = \left[ \begin{array}{cc}
\xi_a & * \\
0 & \xi_a^{-1} \end{array} \right],\]
where 
\[\xi_a  = 
\begin{cases}
1 & \text{if } \, S_{n-1}(v) = 0,\\
\frac{s_1 y - z}{-s_1^{-1}y+z} & \text{if } \, xy S_{n-1}(v) -  (xy-z)  S_{n-2}(v) - z S_n(v) = 0 \text{ and } S_{n-1}(v) \not=  0,
\end{cases}\]
and
\[
*  = 
\begin{cases}
0 & \text{if } \, S_{n-1}(v) = 0,\\
\frac{y(xy-2z)}{xyz-y^2-z^2} & \text{if } \, xy S_{n-1}(v) -  (xy-z)  S_{n-2}(v) - z S_n(v) = 0 \text{ and } S_{n-1}(v) \not=  0.
\end{cases}
\]
Note that there is a small error in \cite{Tran2018IJM} : the canonical longitude in \cite{Tran2018IJM}  is actually the inverse of the canonical longitude. 

If $S_{n-1}(v) = 0$, then we have $\xi_a =1$ and $* = 0$, so $\rho(\lambda_a) = \Id$. 
Since $\rho(a) \not=\Id$, we obtain $\rho(a \lambda^{-l}_a) \not= \Id$, which means that $\rho: \pi_1(S^3 - W_{2n-1}) \to {\rm SL}_2\C$ on the components $v= 2\cos(j\pi/n)$ does not factor through a representation $\pi_1(S^3 - J(2n,2l)) \to {\rm SL}_2\C$.

Since any non-acyclic representation $\rho: \pi_1(S^3 - W_{2n-1}) \to {\rm SL}_2\C$ with  $xy S_{n-1}(v) -  (xy-z)  S_{n-2}(v) - z S_n(v) = 0$ and $S_{n-1}(v) \not=0$ has multiplicity $2$, we obtain the assertion. 
\end{proof} 

\section{$L$-functions of universal deformations} 
\label{sec:L}
In this section, we pursue the study of the $L$-functions of universal deformations, raised from a viewpoint of number theory. 
The main goal is to interpret the ``multiplicity two'' phenomenon for the $L$-functions of the odd-twisted Whitehead links $W_{2n-1}$. 

\subsection{Multiplicity two} 
As we have explained in \Cref{subsec:mult}, the multiplicity may be seen as the valuation in a local ring, and it involves some subtlety. Mazur asked in \cite{Mazur2000} a question on the multiplicity of a certain divisor called the $L$-function for an analogous situation in algebraic geometry over $\C$. 
Morishita and others \cite{KMTT2018} proposed to study its infinitesimal version in the topological setting to observe a finer analogy of the Galois deformation theory in number theory. Continuing their works, we aim to study the cases with coefficients in extensions of $\Z_p$. 

For this purpose, we first paraphrase the definition of ``multiplicity two'' over $\C$ as follows. 
\begin{lemma} \label{lem.mult-two}
Let $f, g\in \C[x,y,z]$, assume that $f$ is irreducible, and let $(A,B,C)$ be a common zero of $f$ and $g$. 
Assume $\partial_z f (A,B,C)\neq 0$ and let $z_f$ denote the implicit function of $f=0$ at $(A,B,C)$. 
Then, the following two properties are equivalent. 

{\rm (i)} The point $(A,B,C)$ is a common zero of the divisors $(f)$ and $(g)$ with multiplicity two. 

{\rm (ii)} The point $(A,B)$ is a zero of $\mca{L}(x,y):=g(x,y,z_f(x,y))$ with multiplicity two in the sense that 
the 1st Taylor polynomial of $\mca{L}(x,y)$ at $(A,B)$ vanishes but the 2nd does not. 
\end{lemma} 

\begin{proof}
By the relationship between the multiplicity of zeros and the length of a module over a local ring (cf.\,\Cref{subsec:mult}), the condition (i) is equivalent to that the module $\C[x,y,z]_{(x-A,y-B,z-C)}/(f,g)$ has length two as a $\C[x,y,z]_{(x-A,y-B,z-C)}$-module. 
Since the length of the module persists under the completion, 
by $\C[\![x-A,y-B,z-C]\!]/(f,g)\cong \C[\![x-A,y-B]\!]/(\mca{L}(x,y))$, 
we obtain the equivalences\\ 
\ \ \ (i) $\iff$ 
${\rm length}_{\C[\![x-A,y-B]\!]} \C[\![x-A,y-B]\!]/(\mca{L}(x,y))=2$\\ 
\ \ \ $\iff$  $\mca{L}(x,y)\in (x-A,y-B)^2$ and $\mca{L}(x,y)\notin (x-A,y-B)^3$ $\iff$ (ii). 
\end{proof}

\subsection{The universal deformations} 
Let $\pi$ denote the group of $W_{2n-1}$. 
Let $\F$ be any field with characteristic $p>2$ and $\mathcal O$ a CDVR with $\mathcal O/\mf{m}_\mathcal O=\F$. 
As before, let $f_n\in \Z[x,y,z]$ denote the reduced polynomial defining the geometric component of the ${\rm SL}_2\C$-character variety, and put $\ol{f_n}:=f$ mod $p$.  
Then each conjugacy class of absolutely irreducible ${\rm SL}_2\F$-representations
corresponds to a point of a component of the character variety $\ol{f_n}(x,y,z)=0$ in $\F^3$. 
(If $\F$ is a finite field or an algebraically closed field, then the inverse correspondence exists,  by the fact that the Brauer group of a finite field is trivial and \cite[Proposition 3.4]{Marche-RIMS2016}.) 

Let $(\ol{A}, \ol{B}, \ol{C}) \in \F^3$ be a zero of $\ol{f_n}$, 
let $(A,B)\in \mathcal O^2$ be a lift of $(\ol{A},\ol{B})$, and put $\mca{R}=\mathcal O[\![x-A,y-B]\!]$. 
Suppose in addition that $\partial_z  f_n (\ol{A}, \ol{B}, \ol{C})\neq 0$. 
Then, Hensel's lemma for multivariable functions yields the implicit function $z_{f_n}(x,y)$ of $ f_n(x,y,z)=0$ around $(A,B)$ and a natural map $\mathcal O[x,y,z]\to \mathcal O[\![x-A,y-B]\!]; z\mapsto z_{f_n}(x,y)$. 

\begin{lemma} If $\ol{\rho}$ is an absolutely irreducible ${\rm SL}_2\F$-representation at $(\ol{A},\ol{B})$, then $\ol{\rho}$ admits a universal deformation $\bs{\rho}:\pi\to \SL_2\mca{R}$ over $\mathcal O$ with $\mca{R}=\mathcal O[\![x-A,y-B]\!]$. 
\end{lemma} 

\begin{proof}
Riley's representation $\rho^R$ is defined by regarding $\rho$ in \cref{subsec:double} as that over a finite extension of $\mca{O}[x,y,z]$;
\[\rho^R(a)=\spmx{\frac{x+\sqrt{x^2-4}}{2}&1\\0&\frac{x-\sqrt{x^2-4}}{2}}, \ \ 
\rho^R(b)=\spmx{\frac{y+\sqrt{y^2-4}}{2}&0\\u&\frac{y-\sqrt{y^2-4}}{2}},\]
$u=z-\frac{x+\sqrt{x\phantom{y\!\!\!}^2-4}}{2}\frac{y+\sqrt{y^2-4}}{2}-\frac{x-\sqrt{x\phantom{y\!\!\!}^2-4}}{2}\frac{y-\sqrt{y^2-4}}{2}.$
Let ${\bs \rho}^R$ denote the image of $\rho^R$ via an extension of Hensel's map. 
Since the trace of ${\bs \rho}^R$ is a ${\rm SL}_2$-character over $\mca{R}=\mathcal O[\![x-A,y-B]\!]$, 
Nyssen and Carayol's theorems (\cite[Theorem 1]{Nyssen1996}, \cite[Theorem 1]{Carayol1994}) 
yield a representation ${\bs \rho}$ over $\mca{R}$ which is strictly equivalent to ${\bs \rho}^R$. 
By \cite[Theorem 2.2.4]{KMTT2018}, this ${\bs \rho}$ is in fact a universal deformation of $\ol \rho$ over $\mathcal O$. 
\end{proof} 

\begin{remark}
A similar construction of ${\bs \rho}$ may be given by using the tautological representation $\rho^T$ introduced in \cite{Benard2020OJM}, as far as $\Tr \rho^T$ may be regarded as an ${\rm SL}_2$-character over $\mca{R}=\mca{O}[\![x-A,y-B,...]\!]$.
\end{remark} 

\subsection{The $L$-functions} 
Following \cite{KMTT2018}, we define \emph{the $L$-function} $L_{\bs \rho}$ of the universal deformation $\bs{\rho}$ of $\ol{\rho}$ to be the order of the 1st twisted homology group $H_1(\pi, \bs{\rho})$. 
We may easily verify that the 0th twisted homology $H_0(\pi, \bs{\rho})$ is trivial. 
\begin{lemma} \label{lem.Ltau} 
The representation ${\bs \rho}$ is rationally acyclic. 
We have 
\[L_{\bs \rho}\doteq \tau_{\bs \rho}=\tau_n(x,y,z_{f_n}(x,y))\neq 0\]
in $\mca{R}=\mca O[\![x-A,y-B]\!]$, where $\tau_n(x,y,z_{f_n}(x,y))$ denotes the image of $\tau_n(x,y,z)$ via Hensel's map $\Z[x,y,z]\to \mca O[\![x-A,y-B]\!]$, 
and $\doteq$ denotes the equality up to multiplication by units. 
\end{lemma} 

\begin{proof}
The first assertion follows from \cref{prop.nonacyclic}. 
Indeed, on the variety of absolutely irreducible representations over ${\rm Frac}\mca{R}$, the torsion function $\tau$ is defined, and its zero locus corresponds to non-acyclic representations over ${\rm Frac}\mca{R}$. 
Since $\bs{\rho}$ is equivalent to the image of the universal representation via Hensel's map, $\tau_{\bs{\rho}}$ is the image of $\tau_n$. 
The evaluation of $\tau$ at ${\bs \rho}$ yields $\tau_{\bs \rho}=\tau_n(x,y,z_{f_n}(x,y))\neq 0$, so ${\bs \rho}$ is rationally acyclic.

Since the link exterior is an Eilenberg--MacLane space, 
a CW complex of the universal cover is homotopic to 
the free resolution $\mathscr{C}_\ast$ of $\Z$ over $\Z[\pi]$ given by the Fox free derivative 
(cf.\,\cite[Proposition I.4.2, Exercise II.5.2(b)]{Bro}). 
Hence, the twisted homology and the twisted Reidemeister torsion of $\bs{\rho}$ are both calculated by 
the twisted complex $V_{\bs{\rho}}\otimes_{\Z[\pi]}\mathscr{C}_\ast$, 
where $V_{\bs{\rho}}=\mca{R}^2$ is the free $\mca{R}$-module regarded as a right $\Z[\pi]$-module via $\bs{\rho}\circ (g\mapsto g^{-1})$, 
and we naturally have $L_{\bs {\rho}}=L_{\bs {\rho}}/1 \doteq \tau_{\bs{\rho}}$ in $\mca{R}$. This completes the proof.  
\end{proof} 

Now we prove the final assertion stated in \Cref{ss.Lfn}. 
\begin{proof}[Proof of \Cref{theo:L}] 
The torsions of residual representations are given by $\ol{\tau_n}:=\tau_n$ mod $p$. 
By \Cref{lem.Ltau}, we have $L_{\bs \rho}\, {\rm mod}\, \mf{m}_\mca{R}\doteq  \ol{\tau_n}$ in $\F[\![x-\ol{A},y-\ol{B}]\!]$.
Thus, we have $L_{\bs \rho}\,\dot{\neq}\, 1$ iff $\ol{\rho}$ is non-acyclic, that is, $\ol{\rho}$ lies in the intersection of $\ol{\tau_n}=0$ and $\ol{f_n}=0$ in $\F^3$. 

Note that \Cref{lem.mult-two} persists if we replace $\C$ by an algebraic closure $\ol{{\rm Frac}\,\mathcal{O}}$ of ${\rm Frac}\,\mathcal{O}$.  
Since the Taylor expansion of $\tau_n(x,y,z_f(x,y))$ is defined by the formal derivatives, the multiplicities of zeros of the $L$-function $L_\bs \rho$ as a formal series coincide with those in \cref{lem.mult-two}. 
Thus, \cref{lem.mult-two} and \cref{theo:TW} on the multiplicities of common zeros of the character variety $f_n$ and the torsion function $\tau_n$ yield the assertion. 
\end{proof}

\begin{remark} 
By the definition, the notion of multiplicity is independent of the choice of coordinates. 
The assumption $\partial_z \ol{f_n} \neq 0$ in \Cref{theo:L} may be replaced by $\partial_x \ol{f_n} \neq 0$ or $\partial_y \ol{f_n} \neq 0$, 
whereas cases with non-smooth reduction would be of particular interest. 
\end{remark} 

\begin{remark}
In the case of twist knots $J(2,2l)$, we have $L_{\bs \rho}\doteq k_n(x)^2$ in $\mca{R}=O[\![x-A]\!]$ for certain series $k_n(x)\in \Z[x]$ which is independent of the choices of resdual representations $\ol{\rho}$ \cite[Theorem E]{TTU}. 
So, a natural question for $W_1$ is whether 
we have $L_{\bs \rho} \doteq k(x,y)^2$ for some $k(x,y)\in \Z[x,y]$ which is independent of $\ol{\rho}$, and to find interpretations for $(k_n(x))_n$.   
Numerical studies of $L_{\bs \rho}$ for $W_{2n-1}$ would be of further interest. 
\end{remark}

\bibliographystyle{amsalpha} 
\bibliography{BTTU.WhiteheadLinks.v2.arXiv.bbl}

\end{document}